\documentclass[11pt,reqno]{amsart}
\usepackage{graphicx, subfigure}
\usepackage{amssymb}
\usepackage{epstopdf}
\usepackage{amssymb}
\usepackage[a4paper, total={6.6in, 9.9in}]{geometry}
\usepackage{multirow}
\usepackage{bm}
\usepackage{amsmath,amsfonts,amsthm,mathrsfs,amssymb}
\usepackage{mathtools}
\usepackage{framed}
\usepackage{exscale}
\usepackage{relsize}
\usepackage{enumerate}
\usepackage{epstopdf}
\usepackage{latexsym}
\usepackage[usenames]{color}
\usepackage{hyperref}
\usepackage{graphics}
\usepackage{framed}
\usepackage{tikz}
\usepackage{enumerate}
\usepackage{xcolor}

\numberwithin{equation}{section}

\newtheorem{definition}{Definition}[section]
\newtheorem{theorem}{Theorem}[section]
\newtheorem{lemma}[theorem]{Lemma}

\newtheorem{corollary}[theorem]{Corollary}

\theoremstyle{definition}

\theoremstyle{remark}
\newtheorem{remark}[theorem]{Remark}

\numberwithin{equation}{section}

\newcounter{saveeqn}

\begin{document}

\title[Spectral properties of transmission eigenmodes and applications]{Spectral properties of surface-localized transmission eigenmodes and applications to inverse scattering problems}

\author{Yan Jiang}
\address{Department of Mathematics, City University of Hong Kong, Hong Kong SAR, China.}
\email{yjian24@cityu.edu.hk}

\author{Hongyu Liu}
\address{Department of Mathematics, City University of Hong Kong, Hong Kong SAR, China.}
\email{hongyu.liuip@gmail.com, hongyliu@cityu.edu.hk}

\author{Kai Zhang}
\address{Department of Mathematics, Jilin University, Changchun, Jilin, China.}
\email{zhangkaimath@jlu.edu.cn}

\author{Haoran Zheng}
\address{School of Mathematical Sciences, South China Normal University, Guangzhou, Guangdong, China.}
\email{hrzheng@scnu.edu.cn}


\thanks{}

\date{}


\maketitle

\begin{abstract}
	This paper investigates a distinctive spectral pattern exhibited by transmission eigenfunctions in wave scattering theory. Building upon the discovery in \cite{CDHLW21,CDLS23} that these eigenfunctions localize near the domain boundary, we derive sharp spectral density estimates—establishing both lower and upper bounds—to demonstrate that a significant proportion of transmission eigenfunctions manifest this surface-localizing behavior. Our analysis elucidates the connection between the geometric rigidity of eigenfunctions and their spectral properties. Though primarily explored within a radially symmetric framework, this study provides rigorous theoretical insights, advances new perspectives in this emerging field, and offers meaningful implications for inverse scattering theory.
\end{abstract}


\section{Physical origin and connection to existing studies}\label{Background}
{In wave scattering theory, an incident probing field $u^i$ interacts with an inhomogeneous medium represented by the parameter pair $(\Omega, \mathbf{n})$, where $\Omega:=\operatorname{supp}(\mathbf{n}-1)$ denotes the spatial support of medium inhomogeneity and $\mathbf{n}$ characterizes the refractive index distribution.} This interaction generates a total field $u$ and scattered field $u^s$, related through:
\begin{equation*}
	u^s = u - u^i,
\end{equation*}
where the incident field $u^i$ constitutes an entire solution to the Helmholtz equation:
\begin{equation*}
	\left(\Delta+k^2\right) u^i=0 \quad \text { in } \mathbb{R}^N, \quad\quad N = 2,3,
\end{equation*}
with $k$ representing the normalized angular wavenumber. Then the scattering process is governed by the following boundary value problem:
\begin{equation}\label{eq:scattering}
	\begin{cases}
		\Delta u+k^2 \mathbf{n}^2 u=0 \quad \text { in } \mathbb{R}^N, \\
		\lim\limits_{r \rightarrow+\infty} r^{(N-1) / 2}\left(\partial_r u^s-\mathrm{i} k u^s\right)=0,
	\end{cases}
\end{equation}
where $\mathrm{i} = \sqrt{-1}$, $r = |x|$, and $\partial_r u := \hat{x}\cdot\nabla_x u$ with $\hat{x} = x/|x| \in \mathbb{S}^{N-1}$.
According to the results in \cite{CK19,LL23}, the scattering problem \eqref{eq:scattering} admits a unique solution $ u \in H_{loc}^1(\mathbb{R}^N)$, which satisfies the following asymptotic behavior:
\begin{equation*}
	u^s(x)=\frac{e^{\mathrm{i} k r}}{r^{(N-1) / 2}} u_{\infty}(\hat{x})+\mathcal{O}\left(\frac{1}{r^{(N+1) / 2}}\right), \quad \mbox{as} \quad  r\rightarrow+\infty.
\end{equation*}
The function $u_{\infty}(\hat{x})$ is referred to as the far-field pattern, which encapsulates the information of how the incident wave $u^i$ is perturbed by the presence of the scatterer characterized by the parameters $(\Omega , \mathbf{n})$.

The inverse scattering problem, which holds substantial importance in industrial applications, focuses on reconstructing the scatterer's configuration $(\Omega, \mathbf{n})$ from far-field measurements $u_\infty(\hat{x})$. Mathematically, this reconstruction problem admits the nonlinear operator formulation:
\begin{equation}\label{inverse}
	\mathcal{F}(\partial\Omega) = u_\infty(\hat{x}),
\end{equation}
where the operator $\mathcal{F}$ encodes the physical scattering process governed by system \eqref{eq:scattering}. A particularly interesting phenomenon in scattering theory is the non-scattering condition (commonly termed \emph{invisibility/transparency}), characterized by the vanishing far-field pattern $u_\infty \equiv 0$. Under this condition, Rellich's Uniqueness Theorem \cite{CK19} establishes that the scattered field must identically vanish in the exterior domain:
\begin{equation*}
	u^s(x) = 0 \quad \forall x \in \mathbb{R}^N \setminus \overline{\Omega}.
\end{equation*}
Based on the above observation, if invisibility/transparency occurs, the total field $\left.u\right|_{\Omega} \in H^1(\Omega)$ and the incident field $v:=\left.u^i\right|_{\Omega} \in H^1(\Omega)$ fulfill the transmission eigenvalue problem \eqref{eq:trans1}. This result has profound implications for inverse problems and cloaking applications \cite{ABGKLW15,CK19,LL23}.

Invisibility and reconstruction are two sides of the same coin, both fundamentally linked to the transmission eigenvalue problem, which is a vibrant research area in applied mathematics driven by significant theoretical value and practical applications in inverse scattering theory. The spectral properties of transmission eigenvalues play a pivotal role in several key domains such as material parameter identification, computational reconstruction schemes, and cloaking device design. The discrete distribution of transmission eigenvalues has been established in \cite{BCH11,CK10,LV12,NN17,S12}, while results on the location of transmission eigenvalues can be found in \cite{CGH10,LC12,V15,V17}. In addition, substantial work has been devoted to related topics such as the Weyl law and the completeness of generalized eigenfunctions \cite{BP13,FN23,LV12b,LV15,NN21,R16}. These properties have been extensively studied in the literature, and we refer to \cite{CCG10,CCH22,CK19} for more comprehensive surveys on this topic.

Conversely, the challenge of reconstruction has opened a new dimension: while previous research mainly focused on transmission eigenvalues, the investigation of transmission eigenfunctions has revealed a wealth of geometric structures both local and global that provide critical insights for reconstruction. These findings have subsequently been exploited in two main directions:

\begin{itemize}
	\item[(1)] \textbf{Local Geometric Properties} {The works} \cite{BLLW17,BL17,BL21} revealed that transmission eigenfunctions tend to vanish near corners or points of high curvature on $\partial \Omega$. This discovery of local geometric patterns sparked extensive research, leading to similar vanishing phenomena across various physical settings \cite{DDL22, DCL21}. Building on these local geometric properties, a novel inverse scattering scheme is developed which effectively achieves reconstruction for media with cusp singularities \cite{LLL17}.
	
	\item[(2)] \textbf{Global Geometric Structures} Alongside these local geometric properties, recent studies have revealed striking global geometric structures of transmission eigenfunctions. In \cite{CDHLW21}, it is shown that there exists an infinite sequence of transmission eigenfunctions whose $L^2$-energy concentrates on $\partial \Omega$. Such geometrical property is subsequently exploited to develop a super-resolution wave imaging scheme associated with the acoustic scattering. Later, in \cite{CDLS23, DJLZ22,JLZZ22}, new type of surface-localized transmission eigenmodes are observed for the acoustic transmission eigenvalue problem. These surface-localizing phenomena have since been extended to other physical models, including electromagnetic, acoustic-elastic, and elastic transmission problems \cite{DLWW22, DLLT23, FD23,JLZZ23}. Such global geometric properties have been successfully leveraged to develop super-resolution wave imaging schemes and reconstruct sound-soft and sound-hard obstacles \cite{HLW23,LLWW19}.
\end{itemize}

The present article continues the study of the global properties of transmission eigenfunctions in the context of wave scattering. The results of this paper reveal profound phenomena in addressing the reconstruction problem \eqref{inverse}. Although real transmission eigenvalues imply a vanishing far-field pattern, a state of invisibility, our findings demonstrate that the surface-localization of the corresponding eigenfunctions enables us to locate the scatterer's boundary. By identifying these eigenfunctions and using specialized methods to map their concentrated energy, the boundary's position can be determined.

In fact, some existing results in this direction have utilized the inside-outside duality to determine the transmission eigenvalues \cite{LR15}. Subsequently, by applying the far-field regularization techniques, the eigenfunctions can be reconstructed, thereby enabling the identification of the region through localization of their energies concentration positions. This implies that, although real eigenvalues classically correspond to invisibility, it is important to note that a subset of these eigenvalues are in fact ``pseudo-invisible"  in the sense that the boundary of the region can still be reconstructed via the above procedure.

Moreover, it is known that the ideal form of invisibility should achieve the effect of ``water blending into water", that is, when the refractive index $\mathbf n$ equals 1. Our conclusions, as stated in Theorems \ref{thm:lower_bdd} and \ref{thm:upper_bdd}, provide upper and lower bounds for the density of the surface-localized transmission eigenfunctions. As $\mathbf n$ approaches 1, the lower bound tends toward 0, and the upper bound also decreases, indicating that the likelihood of such pseudo-invisibility diminishes. In other words, the system approaches true invisibility.

Finally, we briefly discuss two potential practical implications of our results:
\begin{itemize}
	\item[(1)] \textbf{Shape Identification}: The surface localization of transmission eigenfunctions suggests that they encode detailed geometric information about the underlying scattering medium. This behavior reflects a form of global geometric rigidity, making surface-localized modes particularly effective for medium characterization \cite{HLLW24, HLW23}. For the inverse acoustic problem \eqref{inverse} discussed above, the general procedure can be roughly described as recovering those trapped transmission eigen-modes by using the exterior wave measurement, and then using the boundary-localizing behaviours to identify $\partial \Omega$. Numerical studies have revealed an intriguing phenomenon:  it is often possible to uniquely reconstruct \((\Omega;\mathbf n)\) from merely a few far‑field measurements $u_{\infty}(\hat x)$. Our results provide a rigorous
	justification on the super-resolution imaging effect that it can produce. Applying our results in Theorems \ref{thm:lower_bdd} and \ref{thm:upper_bdd}, one suggests that for a generic domain $(\Omega, \mathbf{n})$, a single randomly selected incident wave leads to unique identification of the region's shape with high probability.
	\item[(2)] \textbf{Cloaking Design}: In \cite{CDHLW21},  a new interpretation of the invisibility cloaking was proposed, based on the surface localization of acoustic transmission eigenfunctions. According to our results in Theorems \ref{thm:lower_bdd} and \ref{thm:upper_bdd}, there is a high probability that when invisibility/transparency occurs, the propagating wave travels along the surface of the scattering object $(\Omega, \mathbf{n})$, effectively ``sliding" over its surface before resuming its original trajectory. This insight provides a new physical intuition into the design of artificial cloaking devices: achieving surface-resonant or surface-guided wave behavior could be a key mechanism for realizing effective invisibility. Hence, our results offer theoretical support for surface-based cloaking strategies (see \cite{LLLW17} for the corresponding electromagnetic transmission problem).
\end{itemize}

The remainder of this paper is organized as follows. In Section~\ref{sub:main results}, we present the preliminary and main results of this work. In Section~\ref{2Dcase}, we establish a uniform lower bound for $\rho_{\varepsilon,\delta}$ in $\mathbb{R}^2$ and derive an upper bound for $\rho_{\tilde{\varepsilon},\delta}$ dependent on $\delta$. In Section~\ref{3Dcase}, we extend our analysis to three dimensions, investigating both lower bound for $\rho_{\varepsilon,\delta}$ and upper bounds for $\rho_{\tilde{\varepsilon},\delta}$.


\section{Mathematical setup and statement of the main results}\label{sub:main results}

Let $\Omega$ be a unit ball in $\mathbb{R}^N$, $N=2, 3$, and $\mathbf{n}$ be a positive constant with $\mathbf{n} \neq 1$. In this article, we are mainly concerned with the following transmission eigenvalue problem for $u,v \in H^1(\Omega)$:
\begin{equation}\label{eq:trans1}
	\left\{
	\begin{array}{ll}
		\Delta u+k^2\mathbf{n}^2 u=0  &\text{in} \ \Omega,\medskip \\
		\Delta v+k^2 v =0 &\text{in} \ \Omega, \medskip \\
		\displaystyle{u=v,\ \ \frac{\partial u}{\partial\nu}=\frac{\partial v}{\partial\nu} } &\text{on} \ \partial \Omega, \\
	\end{array}
	\right.
\end{equation}
where $\nu$ is the exterior unit normal vector to $\partial\Omega$. It is obvious that $u \equiv v \equiv 0$ are a pair of trivial solutions to \eqref{eq:trans1}. If for a certain $k\in\mathbb{R}_+$, there exists a non-trivial pair of solutions $(u,v)$, then $k$ is called a (real) transmission eigenvalue, and $(u,v)$ is referred to as the corresponding pair of transmission eigenfunctions.

It follows from \cite{CCH22,LC12} that the transmission eigenvalue problem \eqref{eq:trans1} admits a discrete spectrum of infinitely many real eigenvalues, which only accumulate at infinity. Moreover, numerical observations in \cite{CDHLW21}  suggest that a significant portion of the corresponding eigenfunctions exhibit surface localization, in the sense that their $L^2$-energy concentrates in a neighborhood of $\partial \Omega$. Representative numerical examples are presented in Fig. \ref{surface_localized}. It is the aim of this paper to derive a theoretical understanding of this peculiar spectral phenomenon.  In the present analysis, the assumption of radial symmetry serves a structural purpose: it allows us to isolate geometric effects from those due to material contrast and to rigorously examine the role of the refractive index in shaping the spatial behavior of transmission eigenfunctions.
\begin{figure}[h]
	\centering
	\includegraphics[width=0.9\linewidth]{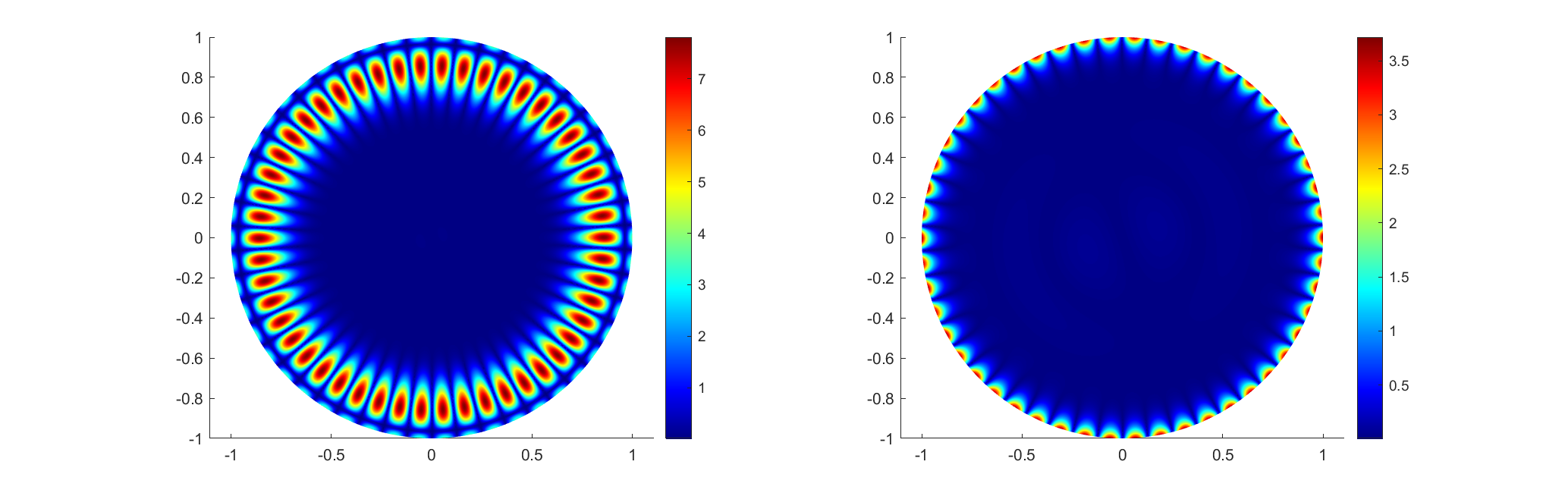}
	\caption{The graph of transmission eigenfunctions $u$(left) and $v$(right) to system \eqref{eq:trans1} associated with $\mathbf{n}=10$ and $k=101.84852668$. }
	\label{surface_localized}
\end{figure}

To begin with, we introduce a qualitative definition of surface-localized function. In what follows, for $\delta \in \mathbb{R}_{+}$, we define
\begin{equation*}
	\mathcal{N}_{\delta}(\partial \Omega):=\{x \in \Omega ; \operatorname{dist}\left(x, \partial \Omega\right)<\delta\},
\end{equation*}
where dist signifies the Euclidean distance in $\mathbb{R}^N, N=2,3$. Clearly, $\mathcal{N}_{\delta}\left(\partial \Omega\right)$ defines an $\delta$-neighborhood of $\partial \Omega$.
\begin{definition}\label{def:varepsilon surface-localized}
	A function $\psi\in L^2\left(\Omega\right)$ is defined as \emph{surface-localized} on $\mathcal{N}_{\delta}\left(\partial \Omega\right)$, if given $\varepsilon>0$ and $\delta>0$, the following inequality holds
	\begin{equation*}
		\mathcal{E}_{\delta}\left(\psi\right) > 1 - \varepsilon,
	\end{equation*}
	where $\mathcal{E}_{\delta}\left(\psi\right):=\frac{\|\psi\|_{L^2\left(\mathcal{N}_{\delta}\left(\partial \Omega\right)\right)}}{\|\psi\|_{L^2\left(\Omega\right)}}$.
\end{definition}
We emphasize that $\varepsilon$ and $\delta$ in Definition \ref{def:varepsilon surface-localized} are independent constants. Meanwhile, from the definition, it is clear that for sufficiently small $\varepsilon$ and $\delta$, the $L^2$-energy of $\psi$ concentrates on the boundary.

To quantitatively characterize the density of surface-localized transmission eigenfunctions, we introduce the following counting functions. First, we define the counting function for real positive ITEs of \eqref{eq:trans1} less than $R$, counted with geometric multiplicities, as:
\begin{equation*}
	N(R) := \sharp \Big\{ (u,v,k) : k \text{ is an ITE},\  0 < k < R \Big\}.
\end{equation*}
Furthermore, to measure the number of ITEs below
$R$ for which at least one associated eigenfunction is surface-localized, we introduce the set  $E^c\left(R, \varepsilon, \delta \right)$, whose elements are triples $\left(u,v,k\right)$ satisfying the following conditions:
\begin{itemize}
	\item[(i)] $k$ is an interior transmission eigenvalue (ITE) of system \eqref{eq:trans1} with $0<k<R$;
	\item[(ii)] The corresponding eigenfunctions $u$ or $v$ are surface-localized (with respect to parameters
	$\delta$ and $\varepsilon$).
\end{itemize}
More precisely, the set is given by
\begin{equation}\label{eq:E_old}
	\begin{aligned}
		E^c\left(R, \varepsilon, \delta \right)  = \left\{ \left(u,v,k\right) ; \text{$k$ is an ITE, }0< k < R, \, \mathcal{E}_{\delta}\left(\psi\right) > 1 - \varepsilon, \, \psi = u \text{ or } v\right\}.
	\end{aligned}
\end{equation}
Accordingly, we define the counting function with geometric multiplicities for the set \eqref{eq:E_old} as
\begin{equation}\label{eq:count_E}
	N^c(R, \varepsilon, \delta) := \sharp E^c(R, \varepsilon, \delta).
\end{equation}
Conversely, the set of ITEs for which neither eigenfunction is surface-localized is defined by
\begin{equation}\label{eq:E^uc_old}
	\begin{aligned}
		E^{uc}\left(R, \varepsilon, \delta \right)  = \left\{ \left(u,v,k\right) ; \text{$k$ is an ITE, }0< k < R, \, \mathcal{E}_{\delta}\left(\psi\right) \leq 1 - \varepsilon, \, \psi = u \text{ and } v\right\}.
	\end{aligned}
\end{equation}
\begin{remark}
	If $(u,v,k) \in E^c(R, \varepsilon, \delta)$, then for any $a \in \mathbb{C} \setminus \{0\}$, the triple $(au, av, k)$ also belongs to $E^c(R, \varepsilon, \delta)$. Hence, $E^c(R, \varepsilon, \delta)$ is essentially a quotient set under scalar multiplication. In subsequent sections, we always assume elements are identified up to multiplication by a nonzero constant.
\end{remark}
\begin{remark}
	The elements of these sets are triples, which is essential since the eigenvalue alone does not uniquely determine the eigenfunctions due to possible multiplicities. In other words, transmission eigenvalues are not necessarily simple. For instance, in two dimensions, the eigenfunction pairs $(u,v)$ take the form:
	\begin{equation*}
		u = C_1 J_m(k\mathbf{n} r) e^{\mathrm{i}m\theta}, \quad v = C_2 J_m(k r) e^{\mathrm{i}m\theta}, \quad m \in \mathbb{N}_+,
	\end{equation*}
	where $J_m$ is the $m$-th order Bessel function. However, the pair
	\begin{equation*}
		u = C_1 J_m(k\mathbf{n} r) e^{-\mathrm{i}m\theta}, \quad v = C_2 J_m(k r) e^{-\mathrm{i}m\theta}, \quad m \in \mathbb{N}_+,
	\end{equation*}
	are also eigenfunctions associated with the \textit{same} eigenvalue $k$, illustrating that $k$ is not a simple eigenvalue.
\end{remark}

Our primary interest lies in the spectral density of surface-localized eigenmodes for transmission eigenvalue problem \eqref{eq:trans1}, characterized by the following limit:
\begin{equation*}\label{eq:def_density}
	\rho_{\varepsilon, \delta} :=\liminf\limits_{R\rightarrow\infty}\frac{N^c\left(R,\varepsilon, \delta\right)}{N\left(R\right)}.
\end{equation*}
For notation simplicity, we introduce some auxiliary functions. Define $P^{(N)}: [-1,1] \to \mathbb{R}$ as
\begin{equation}\label{eq:P2P3}
	P^{(N)}(x)=
	\left\{
	\begin{array}{ll}
		\displaystyle \arccos(x) - x \sqrt{1 - x^2},  & N=2,\\
		\displaystyle (1 - x^2)^{3/2},  & N=3.
	\end{array}
	\right.
\end{equation}
Additionally, define $g: \mathcal{D} \to \mathbb{R}$ on the domain $\mathcal{D} = \{(x,y) \in \mathbb{R}^2 : 0 < x < 1, y > 1/x\}$ by
\begin{equation}\label{eq:g}
	g(x,y)=\sqrt{ x^2-\frac{1-x^2}{y^2-1}}.
\end{equation}

The main results are listed in the following theorems.
\begin{theorem}[Lower bound]\label{thm:lower_bdd}
	For given $0<\varepsilon<\frac{1}{2}$, $0<\delta<1$, and refractive index $0<\mathbf{n}<1$,
	we have
	\begin{equation*}
		\rho_{\varepsilon, \delta} > 0.
	\end{equation*}
	Furthermore, the following dimension-specific bounds hold:
	\begin{equation}\label{eq:lower_bdd}
		\rho_{\varepsilon, \delta} \geq
		\left\{
		\begin{array}{ll}
			\displaystyle \frac{2 P^{(2)}(\mathbf{n})}{\pi\left(1-\mathbf{n}^2\right)},  & \mbox{in}~\mathbb{R}^2,\\
			\displaystyle \frac{P^{(3)}(\mathbf{n})}{\left(1-\mathbf{n}^3\right)},  & \mbox{in}~\mathbb{R}^3,
		\end{array}
		\right.
	\end{equation}
	where $P^{(2)}$ and $P^{(3)}$ are specified in \eqref{eq:P2P3}. Crucially, these lower bounds depend solely on $\mathbf{n}$ and are independent of $\varepsilon$ and $\delta$.
\end{theorem}

\begin{theorem}[Upper bound]\label{thm:upper_bdd}
	Given $0 < \mathbf{n} < 1$, $0 < \delta < 1$, and $0<\frac{\delta}{\mathbf{n}\left(1-\delta\right)}<\tilde{\delta}$, define
	\begin{equation*}
		\tilde{\varepsilon}(\delta) = \frac{1}{4} g\left(1-\delta, 1 + \mathbf{n} \tilde{\delta}\right),
	\end{equation*}
	where $g$ is specified in \eqref{eq:g}. Then, we have
	\begin{equation*}\label{eq:upper_bdd}
		\rho_{\tilde{\varepsilon},\delta} \leq
		\left\{
		\begin{array}{ll}
			\displaystyle \frac{2 B_{U,\tilde{\delta}}^{(2)}\left(\mathbf{n}\right)}{\pi\left(1-\mathbf{n}^2\right)}  ,  & \mbox{in}~\mathbb{R}^2,\\
			\displaystyle \frac{B_{U,\tilde{\delta}}^{(3)}\left(\mathbf{n}\right)}{\left(1-\mathbf{n}^3\right)},  & \mbox{in}~\mathbb{R}^3,
		\end{array}
		\right.
	\end{equation*}
	where $B_{U,\tilde{\delta}}^{(2)}$ and $B_{U,\tilde{\delta}}^{(3)}$ are respectively defined as:
	\begin{equation*}
		\begin{aligned}
			&B_{U, \tilde{\delta}}^{(2)}(\mathbf{n})=P^{(2)}\left(\frac{\mathbf{n}}{1+\mathbf{n} \tilde{\delta}}\right)-P^{(2)}\left(\frac{1}{1+\mathbf{n} \tilde{\delta}}\right) \mathbf{n}^2,\\
			&B_{U, \tilde{\delta}}^{(3)}(\mathbf{n})=P^{(3)}\left(\frac{\mathbf{n}}{1+\mathbf{n} \tilde{\delta}}\right)-P^{(3)}\left(\frac{1}{1+\mathbf{n} \tilde{\delta}}\right) \mathbf{n}^3.
		\end{aligned}
	\end{equation*}
	Here, $P^{(2)}$ and $P^{(3)}$ are the auxiliary functions specified in \eqref{eq:P2P3}.
\end{theorem}

\begin{remark}\label{rem:index-transform}
	In Theorems \ref{thm:lower_bdd} and \ref{thm:upper_bdd}, the condition $0<\mathbf{n}<1$ is assumed. However, we would like to point out that the same results also hold for $\mathbf{n}>1$. To see this, consider the following transformation:
	\begin{equation*}
		\tilde{k} = k\mathbf{n}, \quad \tilde{\mathbf{n}} = \mathbf{n}^{-1}, \quad \tilde{u} = v, \quad \tilde{v} = u.
	\end{equation*}
	One thus obtains
	\begin{equation*}
		\left\{
		\begin{array}{ll}
			\Delta \tilde{u}+\tilde{k}^2\tilde{\mathbf{n}}^2 \tilde{u}=0  &\text{in} \ \Omega,\medskip \\
			\Delta \tilde{v}+\tilde{k}^2 \tilde{v} =0 &\text{in} \ \Omega, \medskip \\
			\displaystyle{\tilde{u}=\tilde{v},\ \ \frac{\partial \tilde{u}}{\partial\nu}=\frac{\partial \tilde{v}}{\partial\nu} } &\text{on} \ \partial \Omega, \\
		\end{array}
		\right.
	\end{equation*}
	which has the same form as the original system \eqref{eq:trans1} with $0<\tilde{\mathbf{n}}<1$. Therefore, the conclusions of Theorems \ref{thm:lower_bdd} and \ref{thm:upper_bdd} extend to the case $\mathbf{n}>1$ as well.
\end{remark}
\begin{remark}
	The lower bound in Theorem~\ref{thm:lower_bdd} holds uniformly for all $0<\varepsilon<\frac{1}{2}$ and $\delta>0$, while the upper bound in Theorem~\ref{thm:upper_bdd} depends on $\delta$.
	The above theorems demonstrate that surface-localized transmission eigenfunctions occur with positive density, but this localization is not a universal property of all transmission eigenfunctions.
\end{remark}

\begin{remark}
	The present analysis is concerned with the asymptotic regime $n \to 1$, in which the material contrast dominates and leads to surface-localized transmission eigenfunctions with positive spectral density. It is worth emphasizing that, for general geometries, a comprehensive existence theory for such boundary-localized eigenfunctions—let alone sharp density estimates—is not yet available. Extending these results to non-radially symmetric settings therefore constitutes a significant and challenging open direction for future research.
\end{remark}

\section{Proof of main results: 2D case}\label{2Dcase}
In this section, we analyze the transmission eigenvalue problem \eqref{eq:trans1} for the unit ball $\Omega \subset \mathbb{R}^2$, assuming a constant refractive index $\mathbf{n} \in (0,1)$.

According to \cite{PS14}, every real positive transmission eigenvalue $k$ must satisfy
$F_m(k) = 0$ for some $m \in \mathbb{N}$, where
\begin{equation*}
	F_{m}(x) := \mathbf{n} J_{m}(x) J_{m}^{\prime}(\mathbf{n} x) - J_{m}(\mathbf{n} x) J_{m}^{\prime}(x).
\end{equation*}
This allows the decomposition:
\begin{equation}\label{eq:E_decomposition}
	E^c\left(R, \varepsilon, \delta \right) = \bigcup_{m}^{\infty} E^c_m\left(R, \varepsilon, \delta \right),
\end{equation}
where
\begin{equation}\label{eq:E_m}
	\begin{split}
		&E^c_m\left(R, \varepsilon, \delta \right)\\& = \left\{ (u,v,k) ; k \text{ is a root of }F_m(x),\, 0<k<R, \,\mathcal{E}_{\delta}(\psi) > 1 - \varepsilon, \, \psi = u \text{ or } v\right\}.
	\end{split}
\end{equation}
Similarly, the set $E^{uc}$ decomposes as:
\begin{equation}\label{eq:E^un_decomposition}
	E^{uc}\left(R, \varepsilon, \delta \right) = \bigcup_{m}^{\infty} E^{uc}_m\left(R, \varepsilon, \delta \right),
\end{equation}
where
\begin{equation}\label{eq:E_m^un}
	\begin{split}
		&E^{uc}_m\left(R, \varepsilon, \delta \right)\\& = \left\{ (u,v,k) ; k \text{ is a root of }F_m(x),\, 0<k<R, \,\mathcal{E}_{\delta}(\psi) \leq 1 - \varepsilon, \, \psi = u \text{ and } v \right\}.
	\end{split}
\end{equation}

\begin{remark}
	Note that for certain $\mathbf{n}$, distinct indices $m_1 \neq m_2$ may yield $F_{m_1}$ and $F_{m_2}$ sharing a common root $k$. However, since each index $m$ corresponds to distinct eigenfunction pairs $(u,v)$, different $m$ values produce different solutions even when sharing the same $k$. Therefore, the decompositions in \eqref{eq:E_decomposition} and \eqref{eq:E^un_decomposition} are disjoint unions.
\end{remark}

Similar to \eqref{eq:count_E}, we define counting functions (accounting for geometric multiplicities) for the sets in  \eqref{eq:E^uc_old},  \eqref{eq:E_m} and \eqref{eq:E_m^un} as:
\begin{equation*}
	\begin{split}
		N^{c}_{m}(R ,\varepsilon, \delta)& := \sharp E_m^c(R, \varepsilon, \delta), \\
		N^{uc}(R,\varepsilon, \delta) &:= \sharp E^{uc}(R, \varepsilon, \delta), \,
		N_{m}^{uc}(R ,\varepsilon, \delta) := \sharp E_m^{uc}(R, \varepsilon, \delta).
	\end{split}
\end{equation*}

\begin{remark}\label{rem:property_N_monotonic}
	By definition, the counting functions
	$N^c(R,\varepsilon, \delta)$, $N^c_{m}(R,\varepsilon, \delta)$, $N^{uc}(R,\\ \varepsilon, \delta)$, and $N_{m}^{uc}(R,\varepsilon, \delta)$
	are all monotone increasing with respect to $R$.
\end{remark}


\subsection{A uniform lower bound estimate in 2D}

We first establish a uniform lower bound for $\rho_{\varepsilon, \delta}$. Beginning with $N(R)$, the decomposition from \cite{PS14} yields:
\begin{equation}\label{eq:N_R_decomposition}
	N(R) = \sum_{m=0}^{\infty} N_{m}(R),
\end{equation}
where $N_{m}(R)$ counts roots of $F_{m}(x)$ in $(0,R)$:
\begin{equation}\label{N_m}
	N_{m}(R) :=\sharp\{(u,v,k): \text{$k$ is a root of $F_{m}(x)$},0 < k < R\}.
\end{equation}

\begin{remark}\label{rem:finite_summation}
	The series \eqref{eq:N_R_decomposition} is well-defined. Indeed, according to \cite{AS72,PS14}, when $m$ is sufficiently large, $N_m(R)$ is equal to zero. Hence, the infinite series in the right hand side of \eqref{eq:N_R_decomposition} only have finite nonzero terms.  Without loss of generality, we assume that the largest $m$ such that $N_m(R) > 0$ is $M(R, \mathbf{n})$.
\end{remark}

In fact, $N_m(R)$ is relative to the counting function of the roots of the Bessel functions $J_m(x)$.
Hence, for the Bessel functions $J_m(x)$, we define the counting functions $N^{B}_{m}(R)$ as follows
\begin{equation}
	N^{B}_{m}(R):=\sharp\{k\in\mathbb{R}: \text{$k$ is a root of $J_{m}(x)$},0 < k<R\}.
	\nonumber
\end{equation}

\begin{remark}\label{rem:2D}
	By the definition, $N^B_m(R)$ is monotonically increasing with respect to $R$. Furthermore, the interlacing property of zeros of Bessel functions implies that $N^B_m(R)$ decreases monotonically with respect to $m$ (\cite{AS72}). Additionally, the asymptotic expansion for zeros of $J_0$ from \cite[Section 7.6.5]{OL74}:
	\begin{equation*}
		j_{0,k} = k\pi - \frac{\pi}{4} + \mathcal{O}\left(k^{-1}\right)
	\end{equation*}
	yields the counting function estimate:
	\begin{equation*}
		N^B_0(R) = \frac{R}{\pi} + \mathcal{O}(1).
	\end{equation*}
\end{remark}

For the quantity $N_m(R)$, we recall from \cite[pp.~807]{PS14} the formulation
\begin{equation}\label{eq:N_m_decomposition}
	N_m(R) =
	\begin{cases}
		N^B_0(R) - N^B_0(\mathbf{n} R) + \gamma_0, & m = 0, \\
		2 \big( N^B_m(R) - N^B_m(\mathbf{n} R) + \gamma_m \big), & m \geq 1,
	\end{cases}
\end{equation}
where $\gamma_{m}$ is an error term taking values in $\{0,1\}$ depending on whether
$R$ is exactly a zero of the Bessel function. Note that $\gamma_m=0$ whenever $N_m(R)=0$.

Moreover, the first zero $j_{m,1}$ of $J_m$ satisfies $j_{m,1} > m$ (\cite{AS72}). Hence,
$N_{\lceil R \rceil}^B(R) = 0$, where $\lceil t\rceil$ denotes the smallest integer greater than $t$. Consequently, $N_{\lceil R \rceil}(R) = 0$ and
\begin{equation}\label{eq:M_upper_bound}
	M(R, \mathbf{n}) < \lceil R \rceil.
\end{equation}
This upper bound is independent of the refractive index $\mathbf{n}$ and dimension.

From \eqref{eq:N_m_decomposition} and \eqref{eq:M_upper_bound}, one deduces that
\begin{equation}\label{eq:PS_error_term}
	2\sum^{\infty}\limits_{m=1}\gamma_{m}+\gamma_{0}=\mathcal{O}(R).
\end{equation}
Furthermore, Weyl's law \cite[Th. 1.6.1]{SV97} yields
\begin{equation*}
	N^{B}_{0}(R) + 2\sum_{m=1}^{\infty} N^{B}_{m}(R) = (2\pi)^{-2} \omega^{2}_{2} R^2 + \mathcal{O}(R),
\end{equation*}
where $\omega_2$ is the volume of the unit ball in $\mathbb{R}^2$. Since the series $\sum^{\infty}\limits_{m=0}N_{m}(R)$ has finitely many nonzero terms, we may reorder summations to obtain:
\begin{equation*}
	\begin{aligned}
		N(R) &= \sum^{\infty}\limits_{m=0}N_{m}(R) + \left( 2\sum^{\infty}\limits_{m=1}\gamma_{m}+\gamma_{0}\right) \\
		&= N^{B}_{0}(R)-N^{B}_{0}(\mathbf{n}R) + 2\sum^{\infty}\limits_{m=1}
		\left(N^{B}_{m}(R)-N^{B}_{m}(\mathbf{n}R)\right)+\mathcal{O}(R)\\
		&= \left(N^{B}_{0}(R)+2\sum^{\infty}\limits_{m=1}N^{B}_{m}(R)\right) - \left(N^{B}_{0}(\mathbf{n}R)+2\sum^{\infty}\limits_{m=1}N^{B}_{m}(\mathbf{n}R)\right)+\mathcal{O}(R)\\
		&=(2\pi)^{-2}\omega^{2}_{2}R^2+\mathcal{O}(R)- (2\pi)^{-2}\omega^{2}_{2}(\mathbf{n}R)^2-\mathcal{O}(\mathbf{n}R)+\mathcal{O}(R)\\
		&=(1-\mathbf{n}^2)(2\pi)^{-2}\omega^{2}_{2}R^2+ \mathcal{O}(R).
	\end{aligned}
\end{equation*}

Next, we estimate $N^{c}(R, \varepsilon, \delta)$. The decomposition in \eqref{eq:E_decomposition} yields:
\begin{equation}\label{eq:N_C_R_decomposition}
	N^{c}(R, \varepsilon, \delta) = \sum^{\infty}\limits_{m=0}N_{m}^{c}(R ,\varepsilon, \delta).
\end{equation}
Note that
\begin{equation*}
	N_{m}^{c}(R, \varepsilon, \delta) \leq N_{m}(R).
\end{equation*}
Hence, it follows from Remark \ref{rem:finite_summation} that the series in \eqref{eq:N_C_R_decomposition} is a finite sum and thus well-defined.

Definition \eqref{eq:E_m} reveals that $N^{c}(R,\varepsilon,\delta)$ depends critically on $\mathcal{E}_{\delta}(\cdot)$. We therefore develop estimates for this term.

Owing to the assumption that the domain is the unit ball in $\mathbb{R}^2$, the corresponding eigenfunctions $u$ and $v$ admit the following expression
\begin{equation*}
	u = C_1 J_m(k\mathbf{n}r)\cdot e^{im\theta}, v = C_2 J_m(r)\cdot e^{im\theta},
\end{equation*}
which implies the energy concentration expression:
\begin{equation*}
	\mathcal{E}^2_{\delta}(u) = \frac{\int_{1-\delta}^1 r J_{m}^2\left(k\mathbf{n} r\right) \mathrm{d} r}{\int_{0}^1 r J_{m}^2\left(k\mathbf{n} r\right) \mathrm{d} r}.
\end{equation*}
A straight calculation gives
\begin{equation*}
	\frac{\int_{1-\delta}^1 r J_{m}^2\left(k\mathbf{n} r\right) \mathrm{d} r}{\int_{0}^1 r J_{m}^2\left(k\mathbf{n} r\right) \mathrm{d} r} = 1-\frac{\int_{0}^{1-\delta} r J_{m}^2\left(k\mathbf{n} r\right) \mathrm{d} r}{\int_{0}^1 r J_{m}^2\left(k\mathbf{n} r\right) \mathrm{d} r}.
\end{equation*}
Applying the substitution $\tilde{r} = k\mathbf{n} r$, we obtain:
\begin{equation*}
	\int_{0}^{1} r J_{m}^2\left(k\mathbf{n} r\right) \mathrm{d} r = \frac{1}{k^2\mathbf{n}^2} \int_{0}^{k\mathbf{n}} \tilde{r} J_{m}^2\left(\tilde{r}\right) \mathrm{d} \tilde{r}.
\end{equation*}
Thus the ratio simplifies to:
\begin{equation*}
	\frac{\int_{0}^{1-\delta} r J_{m}^2\left(k\mathbf{n} r\right) \mathrm{d} r}{\int_{0}^1 r J_{m}^2\left(k\mathbf{n} r\right) \mathrm{d} r} = \frac{\int_{0}^{k\mathbf{n}\tau} \tilde{r} J_{m}^2\left(\tilde{r}\right) \mathrm{d} \tilde{r}}{\int_{0}^{k\mathbf{n}} \tilde{r} J_{m}^2\left(\tilde{r}\right) \mathrm{d} \tilde{r}},
\end{equation*}
where  $\tau = 1-\delta$.

Now, we shall establish sufficient conditions for
\begin{equation}\label{eq:key_estimate_2D}
	QI:=\frac{\int_{0}^{k\mathbf{n}\tau} \tilde{r} J_{m}^2(\tilde{r})  \mathrm{d}\tilde{r}}{\int_{0}^{k\mathbf{n}} \tilde{r} J_{m}^2(\tilde{r})  \mathrm{d}\tilde{r}} <  2\varepsilon - \varepsilon^2.
\end{equation}
This requires estimates for the Bessel function $J_m(x)$, which we develop below.
\begin{lemma}\label{lem:estimate_Jm}
	For $0 < \mathbf{n} < 1$, $0 < k < \frac{m}{\mathbf{n}}$, and $0 < \tau < 1$,  we have
	\begin{equation*}
		\frac{J_m(k \mathbf{n} \tau)}{J_m(k \mathbf{n})} \leq \tau^m \exp\left(\frac{m}{2}\left(1-\tau^2\right)\right).
	\end{equation*}
\end{lemma}

\begin{proof}
	Let $\mathcal{J}_m(x)=x^{-m} J_m(x)$. By \cite{KRA14}, for $0<x \leq m+\frac{1}{2}$, one has
	$$
	\frac{\mathcal{J}_m'(x)}{\mathcal{J}_m(x)} \geq \frac{\sqrt{(2 m+1)^2-4 x^2}-2 m-1}{2 x} \geq-\frac{2 x}{2 m+1}.
	$$
	Then,
	$$
	\ln \frac{\mathcal{J}_m(k \mathbf{n})}{\mathcal{J}_m(k \mathbf{n} \tau)} \geq-\int_{k \mathbf{n} \tau}^{k \mathbf{n}} \frac{2 x}{2 m+1} d x=\frac{k^2 \mathbf{n}^2(\tau^2-1)}{2 m+1},
	$$
	which implies
	$$
	\frac{\mathcal{J}_m(k \mathbf{n})}{\mathcal{J}_m(k \mathbf{n} \tau)} \geq \exp\left(\frac{k^2 \mathbf{n}^2(\tau^2-1)}{2 m+1}\right).
	$$
	Recalling that $\mathcal{J}_m(x)=x^{-m} J_m(x)$, we have
	\begin{equation}\label{eq:J_mz}
		\begin{split}
			\frac{J_m(k \mathbf{n}\tau)}{J_m(k \mathbf{n})}& \leq \tau^m \exp\left(\frac{k^2 \mathbf{n}^2(1 - \tau^2)}{2 m+1}\right)\\& \leq \tau^m \exp\left(\frac{m^2(1-\tau^2)}{2m}\right) \leq \tau^m \exp\left(\frac{m}{2}\left(1-\tau^2\right)\right).
		\end{split}
	\end{equation}
\end{proof}

By Lemma \ref{lem:estimate_Jm}, we derive the following estimate for $QI$.
\begin{lemma}\label{lem:estimate_Jm_integral}
	For $0 < \mathbf{n} < 1$, $0 < k  < \frac{m}{\mathbf{n}}$, and $0 < \tau < 1$,
	we have
	\begin{equation}\label{eq:I}
		QI<2(m+1) \tau^{2+2m} \exp\left(m\left(1-\tau^2\right)\right).
	\end{equation}
\end{lemma}

\begin{proof}
	Let $\mathcal{W}_m(x) = J_m^2(x) - J_{m-1}(x)J_{m+1}(x)$. It follows from \cite{AS72} that
	\begin{equation*}
		[x^2 \mathcal{W}_m(x)]' = 2xJ_m^2(x).
	\end{equation*}
	Hence, we obtain
	\begin{equation*}
		QI = \frac {2\int_{0}^{k\mathbf{n}\tau} \tilde{r} J_{m}^2\left(\tilde{r}\right) \mathrm{d} \tilde{r}}{ k^2 \mathbf{n}^2\mathcal{W}_m(k\mathbf{n})}.
	\end{equation*}
	Then, by the monotonicity of $x J_{m}^2$ in the interval $[0,k \mathbf{n} \tau]$, one observes
	\begin{equation*}
		\int_{0}^{k\mathbf{n}\tau} \tilde{r} J_{m}^2\left(\tilde{r}\right) \mathrm{d} \tilde{r} \leq k\mathbf{n}\tau \cdot k\mathbf{n}\tau J_m^2(k\mathbf{n}\tau).
	\end{equation*}
	Thanks to \cite{JB91}, one has
	\begin{equation*}
		\mathcal{W}_m(x) > \frac{J_m^2(x)}{m+1}, \quad m>-1, -\infty<x<\infty.
	\end{equation*}
	Thus, we have
	\begin{equation*}
		QI < \frac{2(m+1)\tau^2 J_m^2(k\mathbf{n}\tau)}{ J_{m}^2(k\mathbf{n})},
	\end{equation*}
	together with \eqref{eq:J_mz} implies the estimate (\ref{eq:I}).
\end{proof}

Based on Lemma \ref{lem:estimate_Jm} and Lemma \ref{lem:estimate_Jm_integral}, we give the sufficient conditions for \eqref{eq:key_estimate_2D}.
\begin{theorem}\label{thm:key_2D}
	For $0 < \mathbf{n} < 1$, $0 < k  < \frac{m}{\mathbf{n}}$, $0 < \varepsilon < \frac{1}{2}$, and $0 < \delta < 1$,  there exists a constant $C(\varepsilon, \delta)$ independent of $k,m, \mathbf{n}$ such that for any
	$m>C(\varepsilon, \delta)$, we have
	\begin{equation*}\label{eq:Ivare}
		QI < 2\varepsilon - \varepsilon^2,
	\end{equation*}
	where $\tau=1-\delta$. Moreover, $C(\varepsilon, \delta)$ is a constant independent of dimension.
\end{theorem}

\begin{proof}
	By Lemma \ref{lem:estimate_Jm_integral}, it suffices to consider
	\begin{equation}\label{eq:m vare}
		2(m+1) \tau^{2+2m}  e^{m(1-\tau^2)} < 2\varepsilon - \varepsilon^2.
	\end{equation}
	Define $f(x) := x^2 e^{1-x^2}$. Direct calculation gives
	\begin{equation*}
		f'(x) = 2x(1 - x^2)e^{1-x^2},
	\end{equation*}
	which is strictly positive for $x \in (0,1)$. Thus $f$ is strictly increasing on $(0,1)$. Since $f(0) = 0$ and $f(1) = 1$, it follows that $0 < f(\tau) < 1$ for $\tau \in (0,1)$.
	Hence, one has
	\begin{equation}\label{eq:f vare}
		2(m+1) \tau^{2+2m} e^{m(1-\tau^2)} < 2(m+1)\left(\tau^2e^{1-\tau^2} \right)^{m+1} = 2(m+1)f^{m+1}(\tau).
	\end{equation}
	
	For any  $0 < \delta < 1$, let $a = \frac{1+f(\tau)}{2f(\tau)}$. Then, it holds that $a > 1$ and $af(\tau) < 1$. Since $x\ln a < a^x$ holds for all $a >1$, one has
	\begin{equation*}
		m+1 < \frac{a^{m+1}}{\ln a},
	\end{equation*}
	which implies
	\begin{equation}\label{eq:f2 vare}
		2(m+1)f^{m+1}(\tau) < \frac{2a^{m+1} f^{m+1}(\tau)}{\ln a}.
	\end{equation}
	
	Take
	\begin{equation*}
		C(\varepsilon, \delta) =  \frac{\ln \left(\left(\varepsilon-\frac{1}{2}\varepsilon^2\right)\ln a\right)}{\ln (a f(\tau))}-1.
	\end{equation*}
	Then, for $m > C(\varepsilon, \delta)$, we have
	\begin{equation*}
		m+1 > \frac{\ln \left(\left(\varepsilon-\frac{1}{2}\varepsilon^2\right)\ln a\right)}{\ln (a f(\tau))}.
	\end{equation*}
	Notice that $af(\tau) < 1$, it holds that $\ln (a f(\tau))< 0 $, which implies
	\begin{equation*}
		(m+1)\ln (a f(\tau)) < \ln \left(\left(\varepsilon-\frac{1}{2}\varepsilon^2\right)\ln a\right).
	\end{equation*}
	Hence, we obtain
	\begin{equation*}
		\frac{2a^{m+1} f^{m+1}(\tau)}{\ln a} < 2\varepsilon - \varepsilon^2.
	\end{equation*}
	This inequality, together with equations \eqref{eq:f vare} and \eqref{eq:f2 vare}, implies the claimed result in equation \eqref{eq:m vare}.
\end{proof}

As a consequence of the above theorem, we obtain the following corollary.
\begin{corollary}\label{cor:Coro1}
	For $0<\mathbf{n}<1$, $0<k<\frac{m}{\mathbf{n}}$, $0<\varepsilon<\frac{1}{2}$, $0<\delta<1$, and $m>C(\varepsilon, \delta)$, we have
	\begin{equation*}
		\mathcal{E}_{\delta}(u)>1-\varepsilon.
	\end{equation*}
\end{corollary}

\begin{remark}
	While our analysis focuses on the eigenfunction $u$, the density $\rho_{\varepsilon,\delta}$ involves both $u$ and $v$.
	A detailed explanation of this reduction will be given in the next subsection.
\end{remark}

Now, we are at the stage to show the proof of Theorem \ref{thm:lower_bdd} for $N=2$.
\begin{proof}
	\textbf{Step 1.} Recalling \eqref{eq:N_C_R_decomposition} and Remark \ref{rem:finite_summation}, we have
	\begin{equation}
		N^c\left(R, \varepsilon, \delta \right) = \sum_{m=0}^{M(R, \mathbf{n})}  N_m^c\left(R, \varepsilon, \delta \right).
		\nonumber
	\end{equation}
	From Remark \ref{rem:property_N_monotonic}, $N_m^c\left(R, \varepsilon, \delta \right) $ is monotone increasing with respect to $R$, and one deduces that
	\begin{equation*}
		N^{c}(R, \varepsilon, \delta) = \sum^{\infty}\limits_{m=0}N_{m}^{c}(R ,\varepsilon, \delta)
		\geq \sum^{M(R,\mathbf{n})}\limits_{m = \lceil C(\varepsilon, \delta) \rceil}N^c_{m}\left(\min \left(R, \frac{m}{\mathbf{n}}\right) ,\varepsilon, \delta\right).
	\end{equation*}
	In light of Corollary \ref{cor:Coro1}, for $m > C(\varepsilon, \delta)$,  we have
	\begin{equation*}
		N^c_{m}\left(\min \left(R, \frac{m}{\mathbf{n}}\right) ,\varepsilon, \delta\right) = N_{m}\left(\min \left(R, \frac{m}{\mathbf{n}}\right)\right).
	\end{equation*}
	In the rest of this proof, we use $C(\varepsilon,\delta):=\max\{C(\varepsilon,\delta),1\}$ for simplicity. Thus, by the above estimates and \eqref{eq:N_m_decomposition}, we obtain
	\begin{equation}
		\begin{aligned}
			N^{c}(R, \varepsilon, \delta)
			& \geq \sum^{M(R,\mathbf{n})}\limits_{m = \lceil C(\varepsilon, \delta) \rceil}N_{m}\left(\min \left(R, \frac{m}{\mathbf{n}}\right)\right)\\
			& \geq 2\sum^{M(R,\mathbf{n})}\limits_{m = \lceil C(\varepsilon, \delta) \rceil} \left(N^{B}_{m}\left(\min \left(R, \frac{m}{\mathbf{n}}\right)\right)-N^{B}_{m}\left(\min \left(\mathbf{n}R, m\right)\right) \right)+\mathcal{O}(R).
			\nonumber
		\end{aligned}
	\end{equation}
	Besides, it can be derived from \cite{AS72} that $j_{m,1}>m$, which further implies that $N^{B}_{m}(\min (\mathbf{n}R, m))=0$. Therefore, we have
	\begin{eqnarray}\label{eq:N^c_lower_bound}
		N^{c}(R, \varepsilon, \delta) & \geq & 2\sum^{M(R,\mathbf{n})}\limits_{m = \lceil C(\varepsilon, \delta) \rceil} N^{B}_{m}\left(\min \left(R, \frac{m}{\mathbf{n}}\right)\right)+\mathcal{O}(R) \nonumber\\
		& \geq & 2 \sum^{M(R,\mathbf{n})}\limits_{m = \lceil C(\varepsilon, \delta) \rceil} \min \left(N_m^B(R), N_m^B\left(\frac{m}{\mathbf{n}}\right)\right)+\mathcal{O}(R).
	\end{eqnarray}
	
	\noindent
	\textbf{Step 2.} According to \cite{AS72}, the Bessel functions satisfy the following equation
	\begin{equation*}
		x^2 \frac{\mathrm{d}^2 }{\mathrm{d} x^2} J_m+x \frac{\mathrm{d} }{\mathrm{d} x} J_m+\left(x^2-m^2\right) J_m=0.
	\end{equation*}
	Set $J_m = u_m(x) \cdot x^{-\frac{1}{2}}$, then the equation for $u_m(x)$ becomes:
	\begin{equation*}\label{eq:bessel}
		u_m'' + \left(1 - \frac{m^2 - \frac{1}{4}}{x^2}\right)u_m = 0.
	\end{equation*}
	Let $q_m(x) = 1 - \frac{m^2 - \frac{1}{4}}{x^2}$. Since the number of zeros of $J_m$ coincides with that of $u_m$, it is sufficient to estimate the number of zeros of $u_m$. Furthermore, as the first zero $j_{m,1}$ of $J_m$ satisfies $j_{m,1} > m$, the number of zeros of $u_m$ in the interval $(0,M)$ coincides with that in $(m,M)$. In fact, $q_m(x)$ is a bounded variation function  in the interval $(m,M)$. Hence, it follows from \cite[pp348, Thm5.2]{H02} and \cite{TG12} that
	\begin{equation}\label{eq:H02_1}
		\begin{aligned}
			\left|\pi N^{B}_{m}(M)-\int_{m}^{M} \sqrt{q_m(x)}\mathrm{d}x\right| &\leq \pi + \frac{1}{4}\int_{m}^{M}\frac{\left|\mathrm{d} q_m(x)\right|}{q_m(x)}.
		\end{aligned}
	\end{equation}	
	By direct computation, we have
	\begin{equation}\label{eq:H02_2}
		\frac{1}{4}\int_{m}^{M}\frac{\left|\mathrm{d} q_m(x)\right|}{q_m(x)} = \mathcal{O}(\ln(m)).
	\end{equation}
	On the other hand, one derives that
	\begin{equation}\label{eq:H02_3}
		\left|\int_{m}^{M} \sqrt{q_m(x)} \mathrm{d}x-\int_{m}^{M}\sqrt{1-\frac{m^2}{x^2}}\mathrm{d}x \right|=\left|\int_{m}^{M}\sqrt{1-\frac{m^2-\frac{1}{4}}{x^2}}\mathrm - \sqrt{1-\frac{m^2}{x^2}}\mathrm{d}x\right| \leq \frac{\pi}{4m}.
	\end{equation}
	Note that
	\begin{equation}\label{eq:H02_4}
		\int_{m}^{M}\sqrt{1-\frac{m^2}{x^2}}\mathrm{d}x = \left( \sqrt{M^2 - m^2} - m\arccos\left(\frac{m}{M}\right) \right).
	\end{equation}
	When $1\leq m<M$, it follows from \eqref{eq:H02_1}-\eqref{eq:H02_4} that
	\begin{equation}\label{eq:N_m^B(R)}
		N^{B}_{m}(R)=\frac{1}{\pi}\left( \sqrt{R^2 - m^2} - m\arccos\left(\frac{m}{R}\right) \right) + \mathcal{O}(\ln R),
	\end{equation}
	and
	\begin{equation}\label{eq:N_m^B(n)}
		N_m^B\left(\frac{m}{\mathbf{n}}\right) = \frac{m}{\pi} \left( \frac{\sqrt{1-\mathbf{n}^2}}{\mathbf{n} } -  \arccos(\mathbf{n})\right) + \mathcal{O}(\ln R).
	\end{equation}
	
	It follows from \eqref{eq:N_m^B(R)}, \eqref{eq:N_m^B(n)}, and Euler-Maclaurin formula that
	\begin{equation}\label{eq:EM_1}
		\begin{split}
			&\sum^{ \lceil R \rceil }\limits_{m = \lceil  \mathbf{n}R \rceil +1} \left( \sqrt{R^2 - m^2} - m\arccos\left(\frac{m}{R}\right) \right)\\& = \int_{\mathbf{n}R}^{R} \sqrt{R^2 - x^2} - x\arccos\left(\frac{x}{R}\right) \mathrm{d} x + \mathcal{O}(R).
		\end{split}
	\end{equation}
	A direct computation gives that
	\begin{equation}\label{eq:EM_2}
		\begin{split}
			&\int \sqrt{R^2 - x^2} - x\arccos\left(\frac{x}{R}\right) \mathrm{d}x\\&= \frac{3}{4} x \sqrt{R^2 - x^2} + \frac{1}{4} R^2 \arcsin\left( \frac{x}{R} \right) - \frac{1}{2} x^2 \arccos\left( \frac{x}{R} \right),
		\end{split}
	\end{equation}
	and
	\begin{equation}\label{eq:EM_3}
		\int_{\mathbf{n}R}^{R} \sqrt{R^2 - x^2} - x\arccos\left(\frac{x}{R}\right) \mathrm{d}x= \left(\left(\frac{1}{4}+\frac{1}{2}\mathbf{n}^2\right)\arccos(\mathbf{n}) - \frac{3}{4}\mathbf{n}\sqrt{1-\mathbf{n}^2}\right)R^2.
	\end{equation}
	
	\noindent
	\textbf{Step 3.} Recalling \eqref{eq:M_upper_bound} and \eqref{eq:N^c_lower_bound}, for $R \gg C(\varepsilon, \delta)$, one has
	\begin{equation}
		\begin{aligned}
			N^{c}(R, \varepsilon, \delta)
			& \geq 2 \sum^{ \lceil R \rceil }\limits_{m = \lceil C(\varepsilon, \delta) \rceil} \min \left(N_m^B(R), N_m^B\left(\frac{m}{\mathbf{n}}\right)\right)+\mathcal{O}(R) \\
			& \geq 2 \sum^{ \lceil \mathbf{n}R \rceil }\limits_{m = \lceil C(\varepsilon, \delta) \rceil} N_m^B\left(\frac{m}{\mathbf{n}}\right) +  2 \sum^{ \lceil R \rceil }\limits_{m = \lceil  \mathbf{n}R \rceil +1} N_m^B(R)+\mathcal{O}(R).
		\end{aligned}
		\nonumber
	\end{equation}
	Notice that $N^B_m(R)$ is monotonically decreasing with respect to $m$ and is monotonically increasing with respect to $R$, we further have
	\begin{equation}\label{eq:s3 Nc1}
		\begin{aligned}
			N^{c}(R, \varepsilon, \delta)
			& \geq 2 \sum^{ \lceil \mathbf{n}R \rceil }\limits_{m = 1} N_m^B\left(\frac{m}{\mathbf{n}}\right)  - 2 \sum^{ \lceil C(\varepsilon, \delta) \rceil  - 1}\limits_{m = 1} N_m^B\left(\frac{m}{\mathbf{n}}\right) +  2 \sum^{ \lceil R \rceil }\limits_{m = \lceil  \mathbf{n}R \rceil + 1} N_m^B(R) + \mathcal{O}(R)  \\
			& \geq 2 \sum^{ \lceil \mathbf{n}R \rceil }\limits_{m = 1} N_m^B\left(\frac{m}{\mathbf{n}}\right)  - 2 \sum^{ \lceil C(\varepsilon, \delta) \rceil -1}\limits_{m = 1} N_0^B(R) +  2 \sum^{ \lceil R \rceil }\limits_{m = \lceil  \mathbf{n}R \rceil + 1} N_m^B(R) + \mathcal{O}(R) \\
			& \geq 2 \sum^{ \lceil \mathbf{n}R \rceil }\limits_{m = 1} N_m^B\left(\frac{m}{\mathbf{n}}\right)  +  2 \sum^{ \lceil R \rceil }\limits_{m = \lceil  \mathbf{n}R \rceil +1} N_m^B(R) - 2 (\lceil C(\varepsilon, \delta) \rceil -1)  \frac{R}{\pi} +\mathcal{O}(R).
		\end{aligned}
	\end{equation}
	Then, by \eqref{eq:EM_1}-\eqref{eq:s3 Nc1}, one deduces that
	\begin{eqnarray*}
		N^{c}(R, \varepsilon, \delta)
		& \geq & 2 \sum^{ \lceil \mathbf{n}R \rceil }\limits_{m = 1} \frac{m}{\pi} \left( \frac{\sqrt{1-\mathbf{n}^2}}{\mathbf{n} } -  \arccos(\mathbf{n})\right) \nonumber \\
		& &+ 2 \sum^{ \lceil R \rceil }\limits_{m = \lceil  \mathbf{n}R \rceil +1} \frac{1}{\pi}\left( \sqrt{R^2 - m^2} - m\arccos\left(\frac{m}{R}\right) \right)+\mathcal{O}(R\ln R) \nonumber \\
		&\geq & \frac{(\mathbf{n}R)^2}{\pi} \left( \frac{\sqrt{1-\mathbf{n}^2}}{\mathbf{n} } -  \arccos(\mathbf{n})\right) + \mathcal{O}(R\ln R) \nonumber\\
		&& + \frac{2R^2}{\pi}\left( \frac{1}{2} \mathbf{n}^2 \arccos(\mathbf{n}) + \frac{1}{4}\arccos(\mathbf{n}) - \frac{3}{4} \mathbf{n} \sqrt{1 - \mathbf{n}^2}\right) + \mathcal{O}(R\ln R)\\
		&= & \frac{R^2}{2\pi}\left(\arccos(\mathbf{n}) - \mathbf{n} \sqrt{1 - \mathbf{n}^2}\right) + \mathcal{O}(R\ln R),
	\end{eqnarray*}
	together with \eqref{eq:N_R_decomposition}, one has
	\begin{equation*}
		\begin{aligned}
			\frac{N^{c}(R,\varepsilon, \delta)}{N(R)} &\geq \frac{\frac{R^2}{2\pi}\left( \arccos(\mathbf{n}) - \mathbf{n} \sqrt{1 - \mathbf{n}^2}\right) + \mathcal{O}(R\ln R)}{\frac{1}{4}(1-\mathbf{n}^2)R^2 + \mathcal{O}(R)}\\
			& \geq \frac{2\left(\arccos(\mathbf{n}) - \mathbf{n} \sqrt{1 - \mathbf{n}^2}\right) }{\pi (1-\mathbf{n}^2)} +\mathcal{O}\left(\frac{\ln R}{R}\right).
		\end{aligned}
	\end{equation*}
	This completes the proof of Theorem \ref{thm:lower_bdd}.
\end{proof}

The first estimate in \eqref{eq:lower_bdd} provides a lower bound for the density in $\mathbb{R}^2$, denoted by $B_{L}^{(2)}(\mathbf{n})$.
As shown in Figure \ref{fig:lower bound 2D}, $B_{L}^{(2)}(\mathbf{n})$ satisfies $B_{L}^{(2)}(0) = 1$ and $B_{L}^{(2)}(1) = 0$, and is monotonically decreasing with respect to $\mathbf{n}$. This implies that surface-localized eigenfunctions become increasingly prevalent as $\mathbf{n}\to 0$, while surface localization becomes less probable as $\mathbf{n} \to 1$.
\begin{figure}[h]
	\centering
	\includegraphics[width=0.6\linewidth]{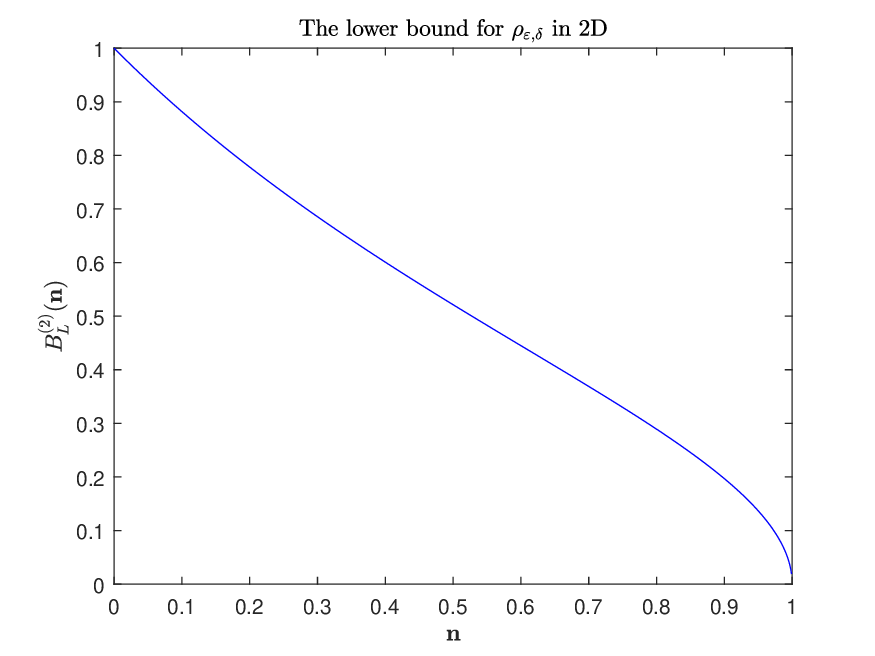}
	\caption{The graph of the lower bound for $\rho_{\varepsilon, \delta}$ in 2D: $B_{L}^{(2)}(\mathbf{n})$.}
	\label{fig:lower bound 2D}
\end{figure}


\subsection{Further comments on the lower bound estimate}
In the previous subsection, our proof of the lower bound relied solely on the transmission eigenfunction $u$. However, the density $\rho_{\varepsilon,\delta}$ is defined using both transmission eigenfunctions $\{u,v\}$. We now address this distinction in detail.

Recall from Corollary \ref{cor:Coro1} that transmission eigenvalues $k$ lie in the interval $I := \bigl(0, \frac{m}{\mathbf{n}}\bigr)$. For our analysis, we partition $I$ into three subintervals:
\begin{equation*}
	I_1 := (0,m), \quad I_2 := (m, (1 + \mu)m), \quad I_3 := \bigl((1 + \mu)m, \frac{m}{\mathbf{n}}\bigr),
\end{equation*}
where $\mu > 0$ is a constant to be specified later. We restrict our analysis to intervals $I_1$ and $I_3$. We note that for interval $I_2$, no estimate for $\mathcal{E}_{\delta}(v)$, which remains an open problem for future investigation.


\subsubsection{Case 1. $k \in I_1$.} Under the conditions in  Corollary \ref{cor:Coro1} and $0 < k < m$, we establish the inequality
\begin{equation}\label{eq:case1}
	\mathcal{E}_{\delta}(u) \geq \mathcal{E}_{\delta}(v),
\end{equation}
which means $u$ is more easily to be surface-localized than $v$.

To begin with,  we introduce the following integral formulas by the recurrence relation of the Bessel function (see formula (9.1.27) in \cite{AS72}).
\begin{lemma}\label{lem:integral_formulas}
	For any $m\in\mathbb{N}$, $ 0<\tau<1$, and $k>0$, one has
	\begin{equation*}
		\int_0^\tau\left(k^2 J_m^{' 2}(k r) r+m^2 \frac{J_m^2(k r)}{r}\right) \mathrm{d} r=\int_0^\tau k^2 J_{m-1}^2(k r) r \mathrm{~d} r-m J_m^2(k \tau),
	\end{equation*}
	and for any $\tau>0$, it holds that
	\begin{equation*}
		\int_0^\tau r J_m^2(r) \mathrm{d} r=\frac{1}{2} \tau^2 J_m^{' 2}(\tau)+\frac{1}{2}\left(\tau^2-m^2\right) J_m^2(\tau)  .
	\end{equation*}
\end{lemma}

In $\mathbb{R}^2$, following \cite{PS14}, the transmission eigenfunctions take the form
\begin{equation*}
	u = C_1 J_m(k\mathbf{n} r) \cdot e^{im\theta},\quad
	v = C_2 J_m(k r) \cdot e^{im\theta}, \quad m\in \mathbb{N}_+.
\end{equation*}
Consequently,
\begin{equation}\label{eq:u and v}
	\mathcal{E}^2_{\delta}(u) = 1 -  	\frac{\int_{0}^{1-\delta} r J_{m}^2\left(k\mathbf{n} r\right) \mathrm{d} r}{\int_{0}^1 r J_{m}^2\left(k\mathbf{n} r\right) \mathrm{d} r},\quad
	\mathcal{E}^2_{\delta}(v) = 1 -  	\frac{\int_{0}^{1-\delta} r J_{m}^2\left(k r\right) \mathrm{d} r}{\int_{0}^1 r J_{m}^2\left(k r\right) \mathrm{d} r}.
\end{equation}
Defining the auxiliary function
\begin{equation}
	\phi(\kappa) :=  \frac{\int_{0}^{1-\delta} r J_{m}^2\left(\kappa r\right) \mathrm{d} r}{\int_{0}^1 r J_{m}^2\left(\kappa r\right) \mathrm{d} r},
	\nonumber
\end{equation}
we obtain the following lemma.
\begin{lemma}\label{lem:monotonic}
	For any given $0<\delta<1$ and $m\in \mathbb{N}_+$, it holds that $\phi(\kappa)$ is monotonically increasing for $\kappa \in (0, j_{m,1})$.
\end{lemma}
\begin{proof}
	Take $\tilde{r} = \kappa r$, then
	\begin{equation*}
		\int_{0}^{1} r J_{m}^2\left(\kappa r\right) \mathrm{d} r = \frac{1}{\kappa^2} \int_{0}^{\kappa} \tilde{r} J_{m}^2\left(\tilde{r}\right) \mathrm{d} \tilde{r}.
	\end{equation*}
	For convenience, we introduce the following notations:
	\begin{equation}\label{eq:fk_def}
		f(\kappa) := \int_{0}^{\kappa} \tilde{r} J_{m}^2\left(\tilde{r}\right) \mathrm{d} \tilde{r},
	\end{equation}
	then
	\begin{equation*}
		\phi(k) = \frac{\int_{0}^{1-\delta} r J_{m}^2\left(\kappa r\right) \mathrm{d} r}{\int_{0}^1 r J_{m}^2\left(\kappa r\right) \mathrm{d} r}=\frac{f(\kappa\tau)}{f(\kappa)},
	\end{equation*}
	where $\tau = 1-\delta$. In view of Lemma \ref{lem:integral_formulas}, we obtain
	\begin{equation}\label{eq:fk_expansion}
		f(\kappa)=\frac{1}{2} \kappa^2 J_m^{' 2}(\kappa)+\frac{1}{2}\left(\kappa^2-m^2\right) J_m^2(\kappa).
	\end{equation}
	Taking the derivative of $\phi(\kappa)$ with respect to $\kappa$ gives
	\begin{equation*}
		\phi'(\kappa) = \frac{[f(\kappa\tau)]'f(\kappa)-[f(\kappa)]'f(\kappa\tau)}{[f(\kappa)]^2} = \frac{f(\kappa\tau)}{f(\kappa)}\left(\frac{[f(\kappa\tau)]'}{f(\kappa\tau)} - \frac{[f(\kappa)]'}{f(\kappa)}\right).
	\end{equation*}
	Let $\tilde{\phi}(\kappa) = \frac{[f(\kappa)]'}{f(\kappa)}$. It follows from \eqref{eq:fk_def} that
	\begin{equation*}
		f'(\kappa) = \kappa J_m^2(\kappa),
	\end{equation*}
	together with \eqref{eq:fk_expansion} implies that
	\begin{equation*}
		\tilde{\phi}(\kappa) = \frac{2 \kappa J_m^2(\kappa)}{ \kappa^2 J_m^{' 2}(\kappa)+\left(\kappa^2-m^2\right) J_m^2(\kappa)}
		= \frac{2}{\kappa} \cdot \frac{1}{\frac{J_m^{' 2}(\kappa)}{J_m^2(\kappa)} + 1 - \frac{m^2}{\kappa^2}}.
	\end{equation*}
	
	Define
	\begin{equation*}\label{eq:psi_def}
		Q(\kappa) =\kappa \cdot \left(\frac{J_m^{' 2}(\kappa)}{J_m^2(\kappa)} + 1 - \frac{m^2}{\kappa^2} \right),\quad
		\psi(\kappa ) = \frac{J_m'(\kappa )}{J_m(\kappa)}.
	\end{equation*}
	Since the Bessel functions satisfy
	\begin{equation*}
		\kappa^2 J_m^{''} + \kappa J_m' + \left(\kappa^2 - m^2\right)J_m = 0,
	\end{equation*}
	we obtain
	\begin{equation*}
		\psi'(\kappa) = \frac{m^2}{\kappa^2} - 1 - \frac{\psi}{\kappa} - \psi^2,
	\end{equation*}
	which implies
	\begin{equation*}\label{eq:x_psi}
		Q(\kappa) = - \psi(\kappa) - \kappa\psi'(\kappa)=\left(-\frac{\kappa J_m'}{J_m}\right)'.	
	\end{equation*}
	Using the Mittag-Leffler expansion (\cite{BKP18}), for $m > 0$ we obtain
	\begin{equation*}
		Q(\kappa) = \left(-\frac{\kappa J_m'}{J_m}\right)' = \sum_{s \geq 1} \frac{4\kappa j_{m,s}^2}{(j_{m,s}^2 - \kappa^2)^2},\quad
		Q'(\kappa) = \sum_{s \geq 1} \frac{4 j_{m,s}^4+12j_{m,s}^2\kappa^2}{(j_{m,s}^2 - \kappa^2)^3}.
	\end{equation*}	
	According to \cite{AS72}, for $0 < \kappa < j_{m,1}$, the function $Q(\kappa)$ is positive and strictly increasing. This monotonicity of $Q$ implies that $\tilde{\phi}$ is strictly decreasing, which consequently implies that $\phi$ is strictly increasing.
\end{proof}

By Lemma \ref{lem:monotonic} and equations \eqref{eq:u and v}, we obtain
\begin{equation*}
	\mathcal{E}_{\delta}(u) = \frac{\|u\|_{L^2(\mathcal{N}_{\delta}(\partial \Omega))}}{\|u\|_{L^2(\Omega)}} \geq \frac{\|v\|_{L^2(\mathcal{N}_{\delta}(\partial \Omega))}}{\|v\|_{L^2(\Omega)}} = \mathcal{E}_{\delta}(v).
\end{equation*}	
This completes the proof of \eqref{eq:case1}.


\subsubsection{Case 3. $k \in I_3$.}
Under the conditions in  Corollary \ref{cor:Coro1}, let $0<\delta<1-\mathbf{n}$, $0<\varepsilon < \frac{1-\delta}{4}$. Then for any given
\begin{equation*}
	\mu > \max\left( \sqrt{\frac{1-4\varepsilon^2}{\left(1-\delta\right)^2 -4\varepsilon^2}}-1, \frac{\delta}{1-\delta}\right),
\end{equation*}	
and wave number $k$ satisfying $(1+\mu)m < k < \frac{m}{\mathbf{n}}$, we have
\begin{equation}\label{eq:case3}
	\mathcal{E}_{\delta}(v) < 1 - \varepsilon.
\end{equation}
This implies that $v$ fails to exhibit surface localization.

Recall that the function $g(x,y)$ defined in \eqref{eq:g} is strictly increasing with respect to $y$ for each fixed $x$. Adopting a proof strategy analogous to \cite[Formula (2.15)]{CDHLW21}, we establish the following lemma.
\begin{lemma}\label{lem:general_inside_estimate}
	For $0 < \delta < 1$, $0<\frac{\delta}{1-\delta} < \mu$, and $\left(1 + \mu\right)m < k$, we have
	\begin{equation}
		1 - \mathcal{E}^2_{\delta}(v) =\frac{\int_{0}^{1-\delta} r J_{m}^2\left(k r\right) \mathrm{d} r}{\int_{0}^1 r J_{m}^2\left(k r\right) \mathrm{d} r} > \frac{1}{2} g(1 - \delta, 1+ \mu),
		\nonumber
	\end{equation}
	and
	\begin{equation}
		g(1 - \delta, 1+ \mu)\rightarrow 0,\quad \text{as} \quad \mu\rightarrow \frac{\delta}{1-\delta}.
		\nonumber
	\end{equation}
\end{lemma}

Lemma \ref{lem:general_inside_estimate} establishes that under the conditions $0<\delta < 1 - \mathbf{n}$ and $0<\varepsilon < \frac{1-\delta}{4}$, for any
$$
\mu > \max\left( \sqrt{\frac{1-4\varepsilon^2}{(1-\delta)^2 -4\varepsilon^2}} - 1,\  \frac{\delta}{1-\delta}\right),
$$
we have
\begin{equation*}
	\mathcal{E}^2_{\delta}(v) < 1 - \frac{1}{2} g(1 - \delta, 1 + \mu) < 1 - 2\varepsilon,
\end{equation*}
which means
\begin{equation*}
	\mathcal{E}_{\delta}(v) < \sqrt{1 - 2\varepsilon} < 1 - \varepsilon.
\end{equation*}

This completes the proof of (\ref{eq:case3}).

Based on the preceding analysis, $u$ generally exhibits stronger surface localization than $v$ (roughly speaking). Therefore, it suffices to consider only $u$ in our subsequent treatment.


\subsection{Upper bound of $\rho_{\tilde{\varepsilon}, \delta}$ in 2D}
In this subsection, we discuss the upper bound of  $\rho_{\tilde{\varepsilon}, \delta}$.

In light of Lemma \ref{lem:general_inside_estimate}, we have following lemma.
\begin{lemma}\label{lem:inside_estimate}
	For $0 < \mathbf{n} < 1$, $0 < \delta < 1$, $0<\frac{\delta}{\mathbf{n}(1 - \delta)}  < \tilde{\delta}$, and $\left(\frac{1}{\mathbf{n}} + \tilde{\delta} \right)m <k$, we have
	\begin{equation*}
		1-\mathcal{E}^2_{\delta}(u) > \frac{1}{2} g\left(1-\delta, 1 + \mathbf{n} \tilde{\delta}\right),\quad
		1-\mathcal{E}^2_{\delta}(v) > \frac{1}{2} g\left(1-\delta, \frac{1 + \mathbf{n} \tilde{\delta}}{\mathbf{n}}\right).
	\end{equation*}
\end{lemma}	

Let
\begin{equation}\label{eq:tilde varepsilon}
	\tilde{\varepsilon}=\frac{1}{4} g\left(1-\delta, 1 + \mathbf{n} \tilde{\delta}\right).
\end{equation}
Then, from $g$ is strictly increasing with respect to $y$ for each fixed $x$, we have
\begin{equation*}
	\mathcal{E}_{\delta}(\psi) < 1 - \tilde{\varepsilon},\quad \psi = u \,\,\, \text{and} \,\,\, v.
\end{equation*}
Under the conditions in Lemma \ref{lem:inside_estimate}, such eigenfunctions cannot be $\tilde{\varepsilon}$-surface-localized on $\mathcal{N}_{\delta}(\partial\Omega)$. In what follows, we focus on estimating the numbers of eigenvalues $k$ that satisfies these conditions; thereby obtain an upper bound for $\rho_{\tilde{\varepsilon}, \delta}$. To this end,
we define the set $\tilde{E}^{uc}(R,\tilde{\varepsilon},\delta)$ as the collection of eigenvalues $k$ that satisfy the conditions in Lemma  \ref{lem:inside_estimate},
$$
\tilde{E}^{uc}(R,\tilde{\varepsilon},\delta) = \left\{ (u,v,k) ; \text{$k$ is an ITE,} \left(\frac{1}{\mathbf{n}} + \tilde{\delta} \right)m <k < R\right\},
$$
and the corresponding counting function is defined as
\begin{equation}
	\tilde{N}^{uc}(R,\tilde{\varepsilon},\delta) :=\sharp \tilde{E}^{uc}(R,\tilde{\varepsilon},\delta) .
	\nonumber
\end{equation}
In the previous analysis, it is not difficult to observe that
\begin{equation*}
	\tilde{N}^{uc}(R,\tilde{\varepsilon},\delta)  \leq N^{uc}(R,\tilde{\varepsilon},\delta).
\end{equation*}
Therefore, in the subsequent analysis, we only need to estimate the counting function  $\tilde{N}^{uc}(R,\tilde{\varepsilon},\delta)$.

In fact, one can decompose $\tilde{E}^{uc}(R,\tilde{\varepsilon},\delta)$ as
\begin{equation}\label{tilde_eq:E_decomposition}
	\tilde{E}^{uc}(R,\tilde{\varepsilon},\delta) = \bigcup_{m}^{\infty} \tilde{E}^{uc}_m\left(R,\tilde{\varepsilon},\delta \right),
	\nonumber
\end{equation}
where
\begin{equation*}
	\tilde{E}^{uc}_m\left(R, \tilde{\varepsilon} ,\delta\right) = \left\{ (u,v,k) ; k \text{ is a root of }F_m(x),\,\left(\frac{1}{\mathbf{n}} + \tilde{\delta} \right)m <k < R\right\}.
\end{equation*}
The corresponding counting function is given by
\begin{equation}\label{tilde_N_uc}
	\tilde{N}^{uc}_m (R,\tilde{\varepsilon},\delta) := \sharp \tilde{E}^{uc}_m\left(R, \tilde{\varepsilon} ,\delta\right).
\end{equation}

Now, we are at the stage to show the proof of Theorem \ref{thm:upper_bdd} for $N=2$.
\begin{proof}
	For convenience, we define
	\begin{equation*}
		C(\mathbf{n}, \tilde{\delta}, R) = \left\lceil\frac{ R }{\left(\frac{1}{\mathbf{n}} + \tilde{\delta}\right)}\right\rceil.
	\end{equation*}
	Observing the analogous structures of the definitions in \eqref{N_m} and \eqref{tilde_N_uc}, we derive an estimate for $\tilde{N}^{uc}(R,\tilde{\varepsilon},\delta)$ from the established estimate for $N(R)$:
	\begin{eqnarray*}\label{eq:I1I2}
		\tilde{N}^{uc} (R,\tilde{\varepsilon},\delta) &=& \sum_{m=0}^{\infty} \tilde{N}^{uc}_m (R,\tilde{\varepsilon},\delta)=\sum_{m=0}^{C(\mathbf{n}, \tilde{\delta}, R)} \tilde{N}^{uc}_m (R,\tilde{\varepsilon},\delta) \\
		&= & \sum_{m=0}^{C(\mathbf{n}, \tilde{\delta}, R)} \left(N_m\left(R\right) - N_m\left(\left(\frac{1}{\mathbf{n}} + \tilde{\delta} \right)m\right)\right)+\mathcal{O}(R) \nonumber \\
		& = &\sum_{m=0}^{C(\mathbf{n}, \tilde{\delta}, R)} N_m\left(R\right) - \sum_{m=0}^{C(\mathbf{n}, \tilde{\delta}, R)}N_m\left(\left(\frac{1}{\mathbf{n}} + \tilde{\delta} \right)m\right)+\mathcal{O}(R) \nonumber \\
		&:=& I_1 - I_2+\mathcal{O}(R).
	\end{eqnarray*}
	
	For $I_1$, \eqref{eq:M_upper_bound} gives
	\begin{eqnarray}\label{eq:I3I4}
		I_1 &=& \sum_{m=0}^{C(\mathbf{n}, \tilde{\delta}, R)} N_m\left(R\right)=2\sum_{m=1}^{C(\mathbf{n}, \tilde{\delta}, R)} \left( N_m^{B}\left(R\right) - N_m^{B}\left(\mathbf{n}R\right)\right) + \mathcal{O}\left( R \right) \nonumber \\
		&=& 2\sum_{m=1}^{C(\mathbf{n}, \tilde{\delta}, R)}  N_m^{B}\left(R\right) - 2\sum_{m=1}^{C(\mathbf{n}, \tilde{\delta}, R)}N_m^{B}\left(\mathbf{n}R\right) + \mathcal{O}\left( R \right) \nonumber \\
		&:=& I_{3} - I_{4} + \mathcal{O}\left( R \right).
	\end{eqnarray}
	Performing a computation analogous to the derivation of \eqref{eq:N^c_lower_bound}, we obtain
	\begin{eqnarray}\label{eq:I3 estimate}
		I_3  &=& 2\sum_{m=1}^{C(\mathbf{n}, \tilde{\delta}, R)}  N_m^{B}\left(R\right) 	\nonumber\\
		&=& 2\sum_{m=1}^{C(\mathbf{n}, \tilde{\delta}, R)}\left( \frac{1}{\pi}\left(\sqrt{R^2-m^2}-m \arccos \left(\frac{m}{R}\right)\right)+\mathcal{O}(\ln R)\right) 	\nonumber \\
		&=& \frac{2}{\pi}P_1\left(\frac{\mathbf{n}}{1+\mathbf{n}\tilde{\delta}}\right) R^2 + \mathcal{O}(R\ln R),
	\end{eqnarray}
	where
	\begin{equation*}
		P_1(x) = \frac{3}{4} x \sqrt{1-x^2}+\frac{1}{4} \arcsin \left(x\right)-\frac{1}{2} x^2 \arccos \left(x\right).
	\end{equation*}
	On the other hand, by \eqref{eq:N_m_decomposition} and \eqref{eq:N_m^B(R)}, one has
	\begin{eqnarray}\label{eq:I4 estimate}
		I_4 & = &  2\sum_{m=1}^{C(\mathbf{n}, \tilde{\delta}, R)} N_m(\mathbf{n}R) \nonumber \\
		& = & 2\sum_{m=1}^{C(\mathbf{n}, \tilde{\delta}, R)}\left( \frac{1}{\pi}\left(\sqrt{\mathbf{n}^2R^2-m^2}-m \arccos \left(\frac{m}{\mathbf{n}R}\right)\right)+\mathcal{O}(\ln R)\right) \nonumber \\
		& = & \frac{2}{\pi}P_1\left(\frac{1}{1+\mathbf{n}\tilde{\delta}}\right) \mathbf{n}^2R^2 + \mathcal{O}(R\ln R).
	\end{eqnarray}
	Combining (\ref{eq:I3 estimate}) and (\ref{eq:I4 estimate}), and substituting into (\ref{eq:I3I4}), we establish
	\begin{equation}
		\begin{aligned}
			I_1 = \frac{2}{\pi}\left( P_1\left(\frac{\mathbf{n}}{1+\mathbf{n}\tilde{\delta}}\right) - P_1\left(\frac{1}{1+\mathbf{n}\tilde{\delta}}\right) \mathbf{n}^2\right)R^2	+ \mathcal{O}(R\ln R).
		\end{aligned}
		\nonumber
	\end{equation}
	For $I_2$, one has
	\begin{equation}
		\begin{aligned}
			I_2 =& \sum_{m=1}^{C(\mathbf{n}, \tilde{\delta}, R)}N_m\left(\left(\frac{1 + \mathbf{n}\tilde{\delta} }{\mathbf{n}}\right)m\right)\\
			=& 2\sum_{m=1}^{C(\mathbf{n}, \tilde{\delta}, R)}\left( N_m^B\left(\left(\frac{1 + \mathbf{n}\tilde{\delta} }{\mathbf{n}}\right)m\right) - N_m^B\left(\left(1 + \mathbf{n}\tilde{\delta} \right)m\right)\right)+\mathcal{O}(R) \\
			=& 2\sum_{m=1}^{C(\mathbf{n}, \tilde{\delta}, R)}\left(\frac{m}{\pi}\left(\sqrt{\left(\frac{1 + \mathbf{n}\tilde{\delta} }{\mathbf{n}}\right)^2-1}-\arccos \left(\frac{\mathbf{n}}{1 + \mathbf{n}\tilde{\delta} }\right)\right)+\mathcal{O}(\ln R)\right)\\
			&- 2\sum_{m=1}^{C(\mathbf{n}, \tilde{\delta}, R)}\left(\frac{m}{\pi}\left(\sqrt{\left(1 + \mathbf{n}\tilde{\delta} \right)^2-1}-\arccos \left(\frac{1}{1 + \mathbf{n}\tilde{\delta} }\right)\right)+\mathcal{O}(\ln R)\right)+\mathcal{O}(R) \\
			=&\frac{1}{\pi}\left(\frac{\mathbf{n}}{1 + \mathbf{n}\tilde{\delta} }\right)^2\left(P_2\left(\frac{\mathbf{n}}{1 + \mathbf{n}\tilde{\delta} }\right) - P_2\left(\frac{1}{1 + \mathbf{n}\tilde{\delta} }\right)\right)R^2 +\mathcal{O}(R\ln R),
		\end{aligned}
		\nonumber
	\end{equation}
	where
	\begin{equation}\label{eq:P2 def}
		P_2(x) = \frac{\sqrt{1 - x^2}}{x} - \arccos \left(x\right).
	\end{equation}
	
	Combining the estimates for $I_1$ and $I_2$, we obtain
	\begin{eqnarray*}
		\displaystyle \tilde{N}^{uc} (R,\tilde{\varepsilon},\delta) &=& \frac{2}{\pi}\left( P_1\left(\frac{\mathbf{n}}{1+\mathbf{n}\tilde{\delta}}\right) - P_1\left(\frac{1}{1+\mathbf{n}\tilde{\delta}}\right) \mathbf{n}^2\right)R^2 \\
		\displaystyle	&-&\frac{1}{\pi}\left(\frac{\mathbf{n}}{1 + \mathbf{n}\tilde{\delta} }\right)^2\left(P_2\left(\frac{\mathbf{n}}{1 + \mathbf{n}\tilde{\delta} }\right) - P_2\left(\frac{1}{1 + \mathbf{n}\tilde{\delta} }\right)\right)R^2+ \mathcal{O}(R\ln R),
	\end{eqnarray*}
	together with
	\begin{eqnarray*}
		\frac{ N^{c} (R,\tilde{\varepsilon},\delta)}{N(R)}  \leq \frac{ N(R) - \tilde{N}^{uc} (R,\tilde{\varepsilon},\delta)}{N(R)},
	\end{eqnarray*}
	implies the conclusion of Theorem \ref{thm:upper_bdd} for $N=2$.
\end{proof}

To facilitate transparent comparison between the bounds, recall the lower bound and upper bound for $N=2$ given by
\begin{equation*}
	\frac{2 P^{(2)} (\mathbf{n})}{\pi(1-\mathbf{n}^2)} \quad \mbox{and} \quad  \frac{2 B_{U,\tilde{\delta}}^{(2)}\left(\mathbf{n}\right)}{\pi\left(1-\mathbf{n}^2\right)},
\end{equation*}
respectively. We explain the process by which the upper bound converges to the lower bound. For a given pair $(\delta, \mathbf{n})$, consider the limiting process as $\tilde{\delta} \to \frac{\delta}{\mathbf{n}(1 - \delta)}$. Then:
\begin{enumerate}
	\item A direct computation shows that $\tilde{\varepsilon} \to 0$.
	
	\item The lower bound remains unchanged throughout this limiting process.
	
	\item The upper bound converges to the lower bound; however, a gap may still remain. Let $\mathcal{G}(\delta)$ denote the limiting difference between the upper and lower bounds:
	\begin{equation*}
		\mathcal{G}(\delta) := \lim_{\tilde{\delta} \to \frac{\delta}{\mathbf{n}(1 - \delta)}} \left( \frac{2 B^{(2)}_{U,\tilde{\delta}}(\mathbf{n})}{\pi(1 - \mathbf{n}^2)} - \frac{2 P^{(2)}(\mathbf{n})}{\pi(1 - \mathbf{n}^2)} \right).
	\end{equation*}
	One can verify that as $\delta \to 0$, the gap $\mathcal{G}(\delta)$ vanishes. This indicates that the convergence is pointwise in $\mathbf{n}$.
	To illustrate, setting $\delta = 0.1$ and $\tilde{\delta} = \frac{2\delta}{\mathbf{n}(1 - \delta)}$, the corresponding lower and upper bounds are illustrated in Figure~\ref{fig:lower and upper bound 2D}.
\end{enumerate}

\begin{figure}
	\centering
	\includegraphics[width=0.6\linewidth]{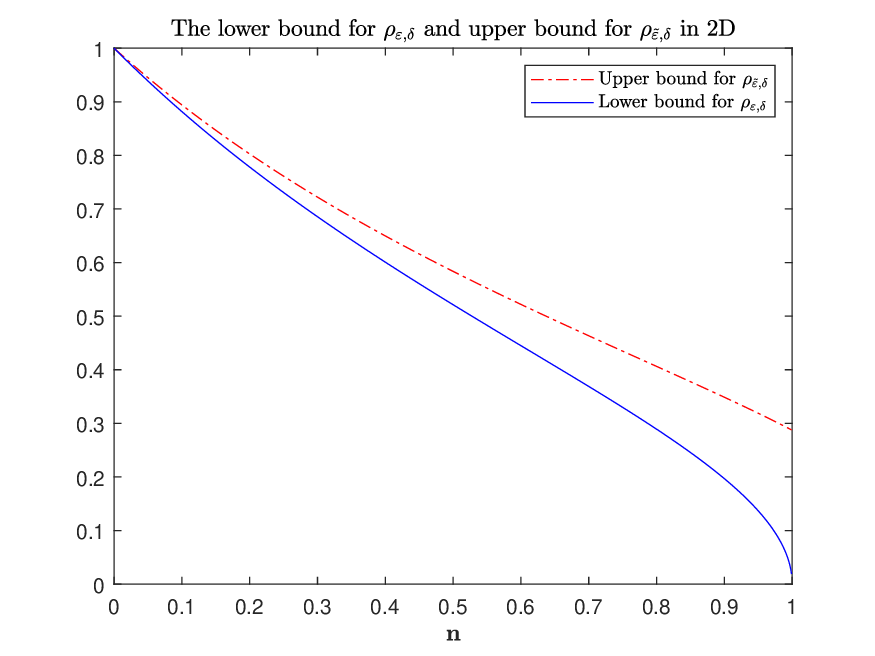}
	\caption{The graph of the lower bound for $\rho_{\varepsilon, \delta}$ and upper bound for $\rho_{\tilde{\varepsilon}, \delta}$ in 2D.}
	\label{fig:lower and upper bound 2D}
\end{figure}


\section{Proof of main results: 3D case}\label{3Dcase}
In what follows, we only sketch the necessary modifications. To avoid ambiguity, in this section we use the superscript ``$(3)$" to denote quantities in the three-dimensional case.
According to \cite{PS14}, in $\mathbb{R}^3$, $k$ is a real positive root of
\begin{equation*}\label{eq:F_3_m}
	F^{(3)}_{m}(x):=\mathbf{n}j_{m}(x) j_{m}'(\mathbf{n} x)-j_{m}(\mathbf{n}x) j_{m}'(x), \text{ for some }m \in \mathbb{N},
\end{equation*}
where $j_m$ is the spherical Bessel functions of order $m$.

Next, we discuss the properties of $F^{(3)}_{m}$. From \cite{AS72}, the spherical Bessel functions $j_m$ satisfy
\begin{equation*}
	x^2 j_m^{\prime\prime}(x) + 2x j_m^{\prime}(x) + \left(x^2 - m(m+1)\right)j_m(x) = 0.
\end{equation*}
Using this equation, a straightforward computation yields the derivative of $F^{(3)}$:
\begin{equation*}
	\left(F^{(3)}_{m}\right)' = -2x^{-1}F^{(3)}_{m} + (1 - \mathbf{n}^2)j_m(x)j_m(\mathbf{n}x).
\end{equation*}
Following this, we turn to the distribution of the zeros of $F^{(3)}$. These zeros correspond to solutions of the equation
\begin{equation}\label{eq:F}
	\mathbf{n}j_{m}(x) j_{m}'(\mathbf{n} x) = j_{m}(\mathbf{n}x) j_{m}'(x).
\end{equation}
\begin{lemma}\label{lem:F_m_intersect}
	For $0 <\mathbf{n}< 1$, at each solution $k_0$ of $H(x) = G(x)$, where
	\begin{equation*}
		H(x) = \frac{j_m'(x)}{j_m(x)}, \quad   G(x) = \mathbf{n} H(\mathbf{n} x),
	\end{equation*}
	we have
	\begin{equation*}
		H'(k_0) < G'(k_0).
	\end{equation*}
\end{lemma}
\begin{proof}
	Near every simple zero $k_0$ of $F^{(3)}_{m}$, we have $j_m\left(k_0\right), j_m\left(\mathbf{n} k_0\right) \neq 0$. Then we can rewrite \eqref{eq:F} near $k_0$ as
	\begin{equation*}
		\mathbf{n} \frac{j_m'(\mathbf{n} x)}{j_m(\mathbf{n} x)}=\frac{j_m'(x)}{j_m(x)}.
	\end{equation*}
	For simplicity, we temporarily omit the subscript $m$ and superscript $(3)$ in $F^{(3)}_{m}$. Defining the function
	\begin{equation*}
		P:=\mathbf{n} \frac{j_m'(\mathbf{n} x)}{j_m(\mathbf{n} x)}-\frac{j_m'(x)}{j_m(x)} = \frac{F}{j_m(x)j_m(\mathbf{n} x)}=\left(1-\mathbf{n}^2\right) \frac{F}{F'+2x^{-1} F},
	\end{equation*}
	we find that at any simple zero $k_0$,
	\begin{equation*}
		P'(k_0)=\left(1-\mathbf{n}^2\right) \frac{F'\left(F'+2x^{-1} F\right)-F\left(F'+2x^{-1} F\right)'}{\left(F'+2x^{-1} F\right)^2}=1-\mathbf{n}^2 .
	\end{equation*}
	This completes the proof.
\end{proof}

In $\mathbb{R}^3$, the eigenfunctions $u$ and $v$ are given by
\begin{equation*}
	u = C_1 j_m(k\mathbf{n}r)\cdot Y_m^l(\hat{x}),\quad v = C_2 j_m(kr)\cdot Y_m^l(\hat{x}),\quad  m \in \mathbb{N}_+, \quad -m\leq l \leq m,
\end{equation*}
where $\hat{x} := x/|x|$ and $Y_m^l$ denotes the spherical harmonic function of degree $m$ and order $l$.

We define the counting function for the real interior transmission eigenvalues (ITEs) of \eqref{eq:trans1} in $\mathbb{R}^3$ (counted with geometric multiplicities) as
\begin{equation*}
	N^{(3)}(R):=\sharp \{(u,v,k): \text{$k$ is an ITE},0< k<R \}.
\end{equation*}
To emphasize that \eqref{eq:E_old} and \eqref{eq:E^uc_old} correspond to the three-dimensional setting, we henceforth denote them by $ E^{(3),c}\left(R, \varepsilon, \delta \right)$ and $E^{(3),uc}\left(R, \varepsilon, \delta \right)$, respectively.

Similar to the two-dimensional case, one can decompose the functions $ E^{(3),c}(R, \varepsilon,\\ \delta)$ and $E^{(3),\mathrm{uc}}(R, \varepsilon, \delta)$ as
\begin{equation}\label{eq:E_decomposition_3D}
	E^{(3),c}\left(R, \varepsilon, \delta \right) = \bigcup_{m}^{\infty}  E_m^{(3),c}\left(R, \varepsilon, \delta \right),\quad E^{(3),uc}\left(R, \varepsilon, \delta \right) = \bigcup_{m}^{\infty} E_m^{(3),uc}\left(R, \varepsilon, \delta \right),
\end{equation}
where
$$\begin{aligned}
	& E_m^{(3),c}\left(R, \varepsilon, \delta \right)\\& = \left\{ (u,v,k) ; k \text{ is a root of }F^{(3)}_m(x),\, 0<k<R, \,\mathcal{E}_\delta(\psi) > 1 - \varepsilon, \, \psi = u \text{ or } v\right\},
\end{aligned}
$$
and
$$ \begin{aligned}&E_m^{(3),uc}\left(R, \varepsilon, \delta \right) \\&= \left\{ (u,v,k) ; k \text{ is a root of }F^{(3)}_m(x),\, 0<k<R, \,\mathcal{E}_\delta(\psi) \leq 1 - \varepsilon, \, \psi = u \text{ and } v\right\}.\end{aligned}
$$
The corresponding counting function $N_{m}^{(3),c}(R, \varepsilon, \delta)$ and $N_{m}^{(3),uc}(R, \varepsilon, \delta)$ can be defined as
\begin{equation*}
	N_m^{(3),c}(R, \varepsilon, \delta) :=\sharp E_m^{(3)}\left(R, \varepsilon, \delta \right),\quad N_m^{(3),uc}(R, \varepsilon, \delta) :=\sharp E_m^{(3),uc}\left(R, \varepsilon, \delta \right).
\end{equation*}

In fact, $N_m^{(3)}(R)$ is relative to the counting function of the roots of the spherical Bessel functions $j_m$.  The corresponding counting function $N^{(3),sB}_{m}(R)$ is defined as follows
\begin{equation*}
	N^{(3),sB}_{m}(R):=\sharp\{k\in\mathbb{R}: \text{$k$ is a root of $j_{m}$},0 < k <R\}.
\end{equation*}
The spherical Bessel function connects to the standard Bessel function through the relation
\begin{equation}\label{eq:j_m}
	j_m(|x|) = \sqrt{\frac{\pi}{2|x|}} J_{m + \frac{1}{2}}(|x|).
\end{equation}
From \eqref{eq:j_m}, we derive
\begin{equation}\label{eq:sB_to_B}
	N^{(3),sB}_{m}(R) =\sharp\{k\in\mathbb{R}: \text{$k$ is a root of $J_{m+\frac{1}{2}}$},0 < k <R\}.
\end{equation}
Denote $N^{(3),\mathrm{sB}}_{m}(R)$ by $N^{B}_{m+\frac{1}{2}}(R)$.
\begin{remark}\label{rem:3D}
	Similar to Remark \ref{rem:2D}, it follows from \cite[Ch. 7.6.5]{OL74} that the $k$-th positive zero $j_{\frac{1}{2},k}$ of $J_{\frac{1}{2}}$ satisfies
	\begin{equation*}
		j_{\frac{1}{2},k} = k\pi  + O(k^{-1}),
	\end{equation*}
	Consequently, the counting function $N^B_{\frac{1}{2}}(R)$ satisfies
	\begin{equation*}
		N^B_{\frac{1}{2}}(R) = \frac{R}{\pi} + O(1), \quad R \to \infty.
	\end{equation*}
\end{remark}

In light of \cite{PS14}, one deduces from Lemma \ref{lem:F_m_intersect}, \eqref{eq:j_m}, and \eqref{eq:sB_to_B}, we have
\begin{equation}\label{N_m_decomposition_3D}
	\begin{aligned}
		N^{(3)}_{m}(R)&=\left(2m+1\right) \left(N^{(3),sB}_{m}(R)-N^{(3),sB}_{m}(\mathbf{n}R) + \gamma^{(3)}_{m}\right)\\
		&=\left(2m+1\right) \left(N^{B}_{m+\frac{1}{2}}(R)-N^{B}_{m+\frac{1}{2}}(\mathbf{n}R) + \gamma^{(3)}_{m}\right),
	\end{aligned}
\end{equation}
where $\gamma_{m}^{(3)}$ is an error term taking values in $\{0,1\}$ depending on whether
$R$ is exactly a zero of the Bessel function.
Similar to \eqref{eq:M_upper_bound} and \eqref{eq:PS_error_term}, we have
\begin{equation}\label{eq:M3}
	M^{(3)}(R, \mathbf{n})<\left\lceil R - \frac{1}{2} \right\rceil,\quad \sum_{m=0}^{\infty}\left(2m+1\right) \gamma_m^{(3)}=\mathcal{O}(R^2).
\end{equation}

In the following subsections, we establish lower and upper bounds for the limit:
\begin{equation*}\label{eq:density_3D}
	\liminf\limits_{R\rightarrow\infty}\frac{N^{(3),c}(R,\varepsilon, \delta)}{N^{(3)}(R)}=\rho_{\varepsilon, \delta}^{(3)}.
\end{equation*}


\subsection{A uniform lower bound estimate in 3D}
We establish a uniform lower bound, with our main result for the three-dimensional case stated in Theorem \ref{thm:lower_bdd}.

Analogous to the two-dimensional case, the asymptotic behavior of $N^{(3)}(R)$ can be derived from the Weyl law \cite[Th. 1.6.1]{SV97}, which gives
\begin{equation}\label{eq:N^3(R)}
	N^{(3)}(R) = (2\pi)^{-3}\omega_3^2(1-\mathbf{n}^3)R^3 + \mathcal{O}(R^2)
\end{equation}
Next, we estimate the counting function $N_m^{(3),\mathrm{c}}(R, \varepsilon, \delta)$. A straightforward calculation yields
\begin{equation*}
	\left(\mathcal{E}^{(3)}_{\delta}(u)\right)^2 = \frac{\int_{1-\delta}^1 r J_{m+\frac{1}{2}}^2\left(k\mathbf{n} r\right) \mathrm{d} r}{\int_{0}^1 r J_{m+\frac{1}{2}}^2\left(k\mathbf{n} r\right) \mathrm{d} r},
\end{equation*}
leading to the following estimates.
\begin{theorem}\label{thm:key_3D}
	For $0 < \mathbf{n} < 1$, $0 < k  < \frac{m + \frac{1}{2}}{\mathbf{n}}$, $0 < \varepsilon < \frac{1}{2}$, and $0 < \delta < 1$,  there exists a constant $C(\varepsilon, \delta)$ independent of $k,m, \mathbf{n}$ such that for any
	$m > C(\varepsilon, \delta),$
	we have
	\begin{equation*}
		\frac{\int_{0}^{1-\delta} r J_{m+\frac{1}{2}}^2\left(k\mathbf{n} r\right) \mathrm{d} r}{\int_{0}^1 r J_{m+\frac{1}{2}}^2\left(k\mathbf{n} r\right) \mathrm{d} r} < 2\varepsilon - \varepsilon^2.
	\end{equation*}
	Moreover, the constant $C(\varepsilon, \delta)$ coincides with its two-dimensional counterpart from Theorem \ref{thm:key_2D}.
\end{theorem}

\begin{corollary}\label{cor:Coro2}
	For $m>C(\varepsilon, \delta), 0 < \mathbf{n} < 1$, $0 < k  < \frac{m+\frac{1}{2}}{\mathbf{n}}$, $0 < \varepsilon < \frac{1}{2}$, and $0 < \delta < 1$, we have
	\begin{equation*}
		\mathcal{E}^{(3)}_{\delta}(u)>1-\varepsilon.
	\end{equation*}
\end{corollary}

Now, we are at the stage to show the proof of Theorem \ref{thm:lower_bdd} for $N=3$.
\begin{proof}
	By \eqref{eq:E_decomposition_3D} and Remark\,\ref{rem:property_N_monotonic}, we have
	\begin{equation*}\label{eq:N^{(3),c}}
		N^{(3),c}(R, \varepsilon, \delta) = \sum^{\infty}\limits_{m=0}N_{m}^{(3),c}(R ,\varepsilon, \delta)\geq \sum^{M(R,\mathbf{n})}\limits_{m = \lceil C(\varepsilon, \delta) \rceil}N^{(3), c}_{m}\left(\min \left(R, \frac{m + \frac{1}{2}}{\mathbf{n}}\right) ,\varepsilon, \delta\right).
	\end{equation*}
	Moreover, for $m>C(\varepsilon, \delta)$, Corollary \ref{cor:Coro2} implies
	\begin{equation*}
		N^{(3),c}_{m}\left(\min \left(R, \frac{m+\frac{1}{2}}{\mathbf{n}}\right) ,\varepsilon, \delta\right) = N^{(3)}_{m}\left(\min \left(R, \frac{m+\frac{1}{2}}{\mathbf{n}}\right)\right).
	\end{equation*}
	Thus, \eqref{N_m_decomposition_3D} yields
	\begin{equation*}
		\begin{aligned}
			&	N^{(3), c}(R, \varepsilon, \delta)\\
			\geq & \sum^{M^{(3)}(R,\mathbf{n})}\limits_{m = \lceil C(\varepsilon, \delta) \rceil}  N_{m}^{(3)}\left(\min \left(R, \frac{m + \frac{1}{2}}{\mathbf{n}}\right)\right)\\
			= &\sum^{M^{(3)}(R,\mathbf{n})}\limits_{m = \lceil C(\varepsilon, \delta) \rceil} (2m+1) \left(N^{B}_{m+\frac{1}{2}}\left(\min \left(R, \frac{m+\frac{1}{2}}{\mathbf{n}}\right)\right)-N^{B}_{m+\frac{1}{2}}\left(\min \left(\mathbf{n}R, m + \frac{1}{2}\right)\right) \right)\\
			&+\mathcal{O}(R^2).
		\end{aligned}
	\end{equation*}
	Furthermore, since $N^{B}_{m+\frac{1}{2}}\left(\min \left(n R, m + \frac{1}{2}\right)\right) = 0$, Remark \ref{rem:3D} and \eqref{eq:M3} imply
	\begin{eqnarray*}\label{eq:N^c_lower_bound_3D}
		&&N^{(3), c}(R, \varepsilon, \delta) \\& \geq& \sum^{\left\lceil R - \frac{1}{2} \right\rceil }\limits_{m = \lceil C(\varepsilon, \delta) \rceil} (2m+1) N^{B}_{m+\frac{1}{2}}\left(\min \left(R, \frac{m+\frac{1}{2}}{\mathbf{n}}\right)\right)+\mathcal{O}(R^2) \nonumber \\
		& \geq&  \sum^{\left\lceil R - \frac{1}{2} \right\rceil}\limits_{m = \lceil C(\varepsilon, \delta) \rceil} (2m+1)\min \left(N_{m+\frac{1}{2}}^B(R), N_{m+\frac{1}{2}}^B\left(\frac{m+\frac{1}{2}}{\mathbf{n}}\right)\right)+\mathcal{O}(R^2)
	\end{eqnarray*}
	
	Analogous to the two-dimensional derivation in \eqref{eq:N_m^B(R)} and \eqref{eq:N_m^B(n)}, we obtain
	\begin{equation}\label{eq:N_m^B(n)_3D}
		N_{m+\frac{1}{2}}^B\left(\frac{m+\frac{1}{2}}{\mathbf{n}}\right) = \frac{m+\frac{1}{2}}{\pi} \left( \frac{\sqrt{1-\mathbf{n}^2}}{\mathbf{n} } -  \arccos(\mathbf{n})\right) + \mathcal{O}(\ln R),
	\end{equation}
	and
	\begin{equation}\label{eq:N_m^B(R)_3D}
		N_{m+\frac{1}{2}}^B(R) = \frac{m+\frac{1}{2}}{\pi}\left( \sqrt{\frac{R^2}{\left(m+\frac{1}{2}\right)^2} - 1} - \arccos\left(\frac{m+\frac{1}{2}}{R}\right) \right) + \mathcal{O}(\ln R).
	\end{equation}
	Applying the monotonicity of $N^B_{m+\frac{1}{2}}(R)$ yields
	\begin{equation*}
		\begin{aligned}
			&N^{(3),c}(R, \varepsilon, \delta)\\
			\geq& \left(\sum^{ \left\lceil \mathbf{n}R - \frac{1}{2} \right\rceil  }\limits_{m = 0} -  \sum^{ \lceil C(\varepsilon, \delta) \rceil - 1}\limits_{m = 0}\right) (2m+1) N_{m+\frac{1}{2}}^B\left(\frac{m+\frac{1}{2}}{\mathbf{n}}\right) \\
			&+  \sum^{ \left\lceil R - \frac{1}{2} \right\rceil }\limits_{m = \left\lceil  \mathbf{n}R  + \frac{1}{2} \right\rceil } (2m+1) N_{m+\frac{1}{2}}^B(R)+\mathcal{O}(R^2)\\
			\geq& \sum^{ \left\lceil \mathbf{n}R - \frac{1}{2} \right\rceil }\limits_{m = 0} (2m+1)N_{m+\frac{1}{2}}^B\left(\frac{m+\frac{1}{2}}{\mathbf{n}}\right)  +   \sum^{ \left\lceil R - \frac{1}{2} \right\rceil }\limits_{m = \left\lceil  \mathbf{n}R  + \frac{1}{2} \right\rceil} (2m+1)N_{m + \frac{1}{2}}^B(R) + \mathcal{O}\left( R^2 \right).
		\end{aligned}
	\end{equation*}
	Combining this with \eqref{eq:N_m^B(n)_3D}, \eqref{eq:N_m^B(R)_3D}, and the Euler-Maclaurin formula yields
	\begin{equation*}
		\begin{aligned}
			&N^{(3),c}(R, \varepsilon, \delta)\\
			\geq & \sum^{ \left\lceil \mathbf{n}R - \frac{1}{2} \right\rceil }\limits_{m = 0} (2m+1)\frac{m+\frac{1}{2}}{\pi} \left( \frac{\sqrt{1-\mathbf{n}^2}}{\mathbf{n} } -  \arccos(\mathbf{n})\right) \\
			& + \sum^{ \lceil R \rceil }\limits_{m = \left\lceil \mathbf{n}R + \frac{1}{2} \right\rceil} (2m+1)\frac{m+\frac{1}{2}}{\pi}\left( \sqrt{\frac{R^2}{(m+\frac{1}{2})^2} - 1} - \arccos\left(\frac{m+\frac{1}{2}}{R}\right) \right)   \\
			& + \mathcal{O}(R^2\ln R) \\
			\geq & \frac{2(\mathbf{n}R)^3}{3\pi} \left( \frac{\sqrt{1-\mathbf{n}^2}}{\mathbf{n} } -  \arccos(\mathbf{n})\right) + \mathcal{O}(R^2\ln R)\\
			& + \frac{2R^3}{9\pi}\left( 3\mathbf{n}^3 \arccos(\mathbf{n}) + \left(1 - 4\mathbf{n}^2\right) \sqrt{1-\mathbf{n}^2} \right) + \mathcal{O}(R^2\ln R)\\
			= & \frac{2R^3}{9\pi}\left( 1 - \mathbf{n}^2\right) ^{\frac{3}{2}} + \mathcal{O}(R^2\ln R).
		\end{aligned}
	\end{equation*}
	From \eqref{eq:N^3(R)}, we obtain
	\begin{equation*}
		\begin{aligned}
			\frac{N^{(3),c}(R,\varepsilon, \delta)}{N^{(3)}(R)} &\geq \frac{\frac{2R^3}{9\pi}\left( 1 - \mathbf{n}^2\right) ^{\frac{3}{2}} + \mathcal{O}(R^2\ln R)}{(2\pi)^{-3}\omega_3^2(1-\mathbf{n}^3)R^3 + \mathcal{O}(R^2\ln R)}\\
			& \geq \frac{ \left( 1 - \mathbf{n}^2\right) ^{\frac{3}{2}} }{ (1-\mathbf{n}^3)} +\mathcal{O}(\frac{\ln R}{R}),
		\end{aligned}
	\end{equation*}
	which establishing the lower bound \eqref{eq:lower_bdd} for three dimensions.
\end{proof}

The second estimate in \eqref{eq:lower_bdd} provides a lower bound for the density in $\mathbb{R}^3$, denoted by $B_{L}^{(3)}(\mathbf{n})$.
As shown in Figure \ref{fig:lower bound 3D}, $B_{L}^{(3)}(\mathbf{n})$ satisfies $B_{L}^{(3)}(0) = 1$ and $B_{L}^{(3)}(1) = 0$, and is monotonically decreasing with respect to $\mathbf{n}$. This implies that surface-localized eigenfunctions become increasingly prevalent as $\mathbf{n}\to 0$, while surface localization becomes less probable as $\mathbf{n} \to 1$.
\begin{figure}[h]
	\centering
	\includegraphics[width=0.8\linewidth]{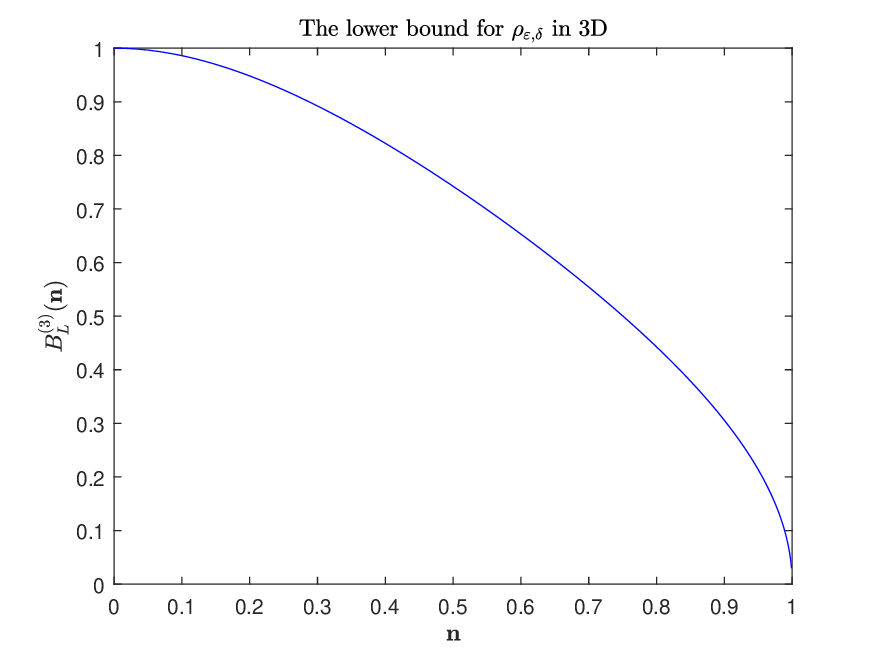}
	\caption{The graph of the lower bound for $\rho_{\varepsilon, \delta}$ in 3D: $B_{L}^{(3)}(\mathbf{n})$.}
	\label{fig:lower bound 3D}
\end{figure}


\subsection{Upper bound of $\rho_{\tilde{\varepsilon},\delta}$ in 3D}
In this subsection, we discuss the upper bound of  $\rho_{\tilde{\varepsilon}, \delta}$ in  three-dimensional case.
First,  we present the following lemma.
\begin{lemma}\label{lem_inside_estimate_3D}
	For $0 < \mathbf{n} < 1$, $0 < \delta < 1$, $0<\frac{\delta}{\mathbf{n}(1-\delta)}<\tilde{\delta}$, and $ \left(\frac{1}{\mathbf{n}} + \tilde{\delta} \right)\left(m+\frac{1}{2}\right) <k$, we have
	\begin{equation*}
		1-\left(\mathcal{E}^{(3)}_{\delta}(u)\right)^2 > \frac{1}{2}g(1-\delta, 1+\mathbf{n} \tilde{\delta}), \quad
		1-\left(\mathcal{E}^{(3)}_{\delta}(v)\right)^2 > \frac{1}{2}g\left(1-\delta, \frac{1+\mathbf{n} \tilde{\delta}}{\mathbf{n}}\right).
	\end{equation*}
\end{lemma}

Recalling (\ref{eq:tilde varepsilon}) and noting that $g$ is strictly increasing in $y$ for each fixed $x$, we have
\begin{equation*}
	\mathcal{E}_{\delta}(\psi) < 1 - \tilde{\varepsilon},\quad \psi = u \,\,\, \text{and} \,\,\, v.
\end{equation*}
Then, we define the set $\tilde{E}^{(3),uc}(R,\tilde{\varepsilon},\delta)$ as the collection of eigenvalues $k$ that satisfy the conditions in Lemma  \ref{lem_inside_estimate_3D},
\begin{equation*}
	\tilde{E}^{(3),uc}(R,\tilde{\varepsilon},\delta) = \left\{ (u,v,k) ; \text{$k$ is an ITE,} \left(\frac{1}{\mathbf{n}} + \tilde{\delta} \right)\left(m+\frac{1}{2}\right) < k < R\right\},
\end{equation*}
and the corresponding counting function is defined as
\begin{equation*}
	\tilde{N}^{(3),uc}(R,\tilde{\varepsilon},\delta) := \sharp \tilde{E}^{(3),uc}(R,\tilde{\varepsilon},\delta) .
\end{equation*}
Our analysis immediately reveals that
\begin{equation*}
	\tilde{N}^{(3),uc}(R,\tilde{\varepsilon},\delta)  \leq N^{(3), uc}(R,\tilde{\varepsilon},\delta).
\end{equation*}
Thus, we need only estimate $\tilde{N}^{(3),\mathrm{uc}}(R,\tilde{\varepsilon},\delta)$ and decompose $\tilde{E}^{(3),\mathrm{uc}}(R,\tilde{\varepsilon},\delta)$ as
\begin{equation*}
	\tilde{E}^{(3),uc}(R,\tilde{\varepsilon},\delta) = \bigcup_{m}^{\infty} \tilde{E}^{(3),uc}_m\left(R,\tilde{\varepsilon},\delta \right),
\end{equation*}
where
\begin{equation*}
	\tilde{E}^{(3),uc}_m\left(R, \tilde{\varepsilon} ,\delta\right) = \left\{ (u,v,k) ; k \text{ is a root of }F^{(3)}_m(x),\,\left(\frac{1}{\mathbf{n}} + \tilde{\delta}_1 \right)\left(m+\frac{1}{2}\right) < k < R\right\}.
\end{equation*}
The corresponding counting function is given by
\begin{equation*}
	\tilde{N}^{(3),uc}_m (R,\tilde{\varepsilon},\delta) := \sharp \tilde{E}^{(3),uc}_m\left(R, \tilde{\varepsilon} ,\delta\right).
\end{equation*}

Now, we are at the stage to show the proof of Theorem \ref{thm:upper_bdd} for $N=3$.
\begin{proof}
	For notational simplicity, we introduce the following notation:
	\begin{equation*}
		C^{(3)}(\mathbf{n}, \tilde{\delta}, R)=\left\lceil\frac{R}{\left(\frac{1}{\mathbf{n}}+\tilde{\delta}\right)}-\frac{1}{2}\right\rceil.
	\end{equation*}
	Therefore, one calculates that
	\begin{eqnarray}\label{eq:N3I1I2}
		&&\tilde{N}^{(3),uc}_m (R,\tilde{\varepsilon},\delta)\notag\\ &=& \sum_{m=0}^{C^{(3)}(\mathbf{n}, \tilde{\delta}, R)} N^{(3)}_m(R)-\sum_{m=0}^{C^{(3)}(\mathbf{n}, \tilde{\delta}, R)} N^{(3)}_m\left(\left(\frac{1}{\mathbf{n}}+\tilde{\delta}\right) \left(m+\frac{1}{2}\right)\right) +\mathcal{O}(R^2) \nonumber \\
		&:=& I^{(3)}_1 - I^{(3)}_2+\mathcal{O}(R^2).
	\end{eqnarray}
	
	For $I^{(3)}_1$, it admits the following decomposition:
	\begin{eqnarray}\label{eq:N3I3I4}
		I^{(3)}_1 & =& \sum_{m=0}^{C^{(3)}(\mathbf{n}, \tilde{\delta}, R)} \left(2m+1\right) N_{m+\frac{1}{2}}^B(R)- \sum_{m=0}^{C^{(3)}(\mathbf{n}, \tilde{\delta}, R)} \left(2m+1\right) N_{m+\frac{1}{2}}^B(\mathbf{n} R)+\mathcal{O}(R^2) \nonumber \\
		& :=&I^{(3)}_3-I^{(3)}_4+\mathcal{O}(R^2).
	\end{eqnarray}
	We further compute $I^{(3)}_3$ and $I^{(3)}_4$ as
	{\small \begin{align}\label{eq:N3I3}
			I^{(3)}_3  =& \sum_{m=0}^{C^{(3)}(\mathbf{n}, \tilde{\delta}, R)} \left(2m+1\right)\left( \frac{m+\frac{1}{2}}{\pi}\left( \sqrt{\frac{R^2}{\left(m+\frac{1}{2}\right)^2} - 1} - \arccos\left(\frac{m+\frac{1}{2}}{R}\right) \right) + \mathcal{O}(\ln R)\right) \nonumber \\
			= & \frac{2}{\pi}P_1^{(3)}\left(\frac{\mathbf{n}}{1+\mathbf{n} \tilde{\delta}}\right)R^3 + \mathcal{O}(R^2 \ln R),
	\end{align}}
	and
	{\small \begin{align}\label{eq:N3I4}
			I^{(3)}_4 = &\sum_{m=0}^{C^{(3)}(\mathbf{n}, \tilde{\delta}, R)} \left(2m+1\right)\left( \frac{m+\frac{1}{2}}{\pi}\left( \sqrt{\frac{\mathbf{n}^2R^2}{\left(m+\frac{1}{2}\right)^2} - 1} - \arccos\left(\frac{m+\frac{1}{2}}{\mathbf{n}R}\right) \right) + \mathcal{O}(\ln R)\right) \nonumber \\
			=&	\frac{2}{\pi}P_1^{(3)}\left(\frac{\mathbf{n}}{1+\mathbf{n} \tilde{\delta}}\right)\mathbf{n}^3R^3 + \mathcal{O}(R^2 \ln R),
	\end{align}}
	where
	\begin{equation*}
		P_1^{(3)}(x) = \frac{1}{9}+\left(\frac{4 x^2 - 1}{9}\right) \sqrt{1-x^2} - \frac{3}{9} x^3 \arccos (x).
	\end{equation*}
	Substituting (\ref{eq:N3I3}) and (\ref{eq:N3I4}) into (\ref{eq:N3I3I4}), we obtain
	\begin{equation}\label{eq:N3I3 est}
		I^{(3)}_1 = \frac{2}{\pi}\left( P^{(3)}_1\left(\frac{\mathbf{n}}{1+\mathbf{n}\tilde{\delta}}\right) - P^{(3)}_1\left(\frac{1}{1+\mathbf{n}\tilde{\delta}}\right) \mathbf{n}^3\right)R^3	+ \mathcal{O}(R^2\ln R).
	\end{equation}
	
	For $I_2^{(3)}$, one has
	\begin{equation}\label{eq:N3I4 est}
		\begin{aligned}
			I^{(3)}_2 =& \sum_{m=1}^{C^{(3)}(\mathbf{n}, \tilde{\delta}, R)}N^{(3)}_m\left(\left(\frac{1 + \mathbf{n}\tilde{\delta} }{\mathbf{n}}\right)\left(m+\frac{1}{2}\right)\right)  \\
			=&\frac{2}{3\pi}\left(\frac{\mathbf{n}}{1 + \mathbf{n}\tilde{\delta} }\right)^3\left(P_2\left(\frac{\mathbf{n}}{1 + \mathbf{n}\tilde{\delta} }\right) - P_2\left(\frac{1}{1 + \mathbf{n}\tilde{\delta} }\right)\right)R^3  +\mathcal{O}(R^2\ln R),
		\end{aligned}
	\end{equation}
	where $P_2$ is given by (\ref{eq:P2 def}).

	Substituting (\ref{eq:N3I3 est}) and (\ref{eq:N3I4 est}) into (\ref{eq:N3I1I2}), we obtain
	\begin{equation*}
		\begin{aligned}
			&\tilde{N}^{(3),uc} (R,\tilde{\varepsilon},\delta)\\
			=& \frac{2}{\pi}\left( P^{(3)}_1\left(\frac{\mathbf{n}}{1+\mathbf{n}\tilde{\delta}}\right) - P^{(3)}_1\left(\frac{1}{1+\mathbf{n}\tilde{\delta}}\right) \mathbf{n}^3\right)R^3 \\
			&- \frac{2}{3\pi}\left(\frac{\mathbf{n}}{1 + \mathbf{n}\tilde{\delta} }\right)^3\left(P_2\left(\frac{\mathbf{n}}{1 + \mathbf{n}\tilde{\delta} }\right) - P_2\left(\frac{1}{1 + \mathbf{n}\tilde{\delta} }\right)\right)R^3+ \mathcal{O}(R^2\ln R),
		\end{aligned}
	\end{equation*}
	together with
	\begin{equation*}
		\frac{ N^{(3),c} (R,\tilde{\varepsilon},\delta)}{N^{(3)}(R)}  \leq \frac{ N^{(3)}(R) - \tilde{N}^{(3),uc} (R,\tilde{\varepsilon},\delta)}{N^{(3)}(R)},
	\end{equation*}
	implies the conclusion of Theorem \ref{thm:upper_bdd} for $N=3$.
\end{proof}

To facilitate transparent comparison between the bounds, recall the lower bound and upper bound for $N=3$ given by
\begin{equation*}
	\frac{2 P^{(3)} (\mathbf{n})}{\pi(1-\mathbf{n}^2)} \quad \mbox{and} \quad  \frac{2 B_{U,\tilde{\delta}}^{(3)}\left(\mathbf{n}\right)}{\pi\left(1-\mathbf{n}^2\right)},
\end{equation*}
respectively. For a given pair $(\delta, \mathbf{n})$, consider the limiting process as $\tilde{\delta} \to \frac{\delta}{\mathbf{n}(1 - \delta)}$. Then:
\begin{enumerate}
	\item A direct computation shows that $\tilde{\varepsilon} \to 0$.
	
	\item The lower bound remains unchanged throughout this limiting process.
	
	\item Let $\mathcal{G}^{(3)}(\delta)$ denote the limiting difference between the upper and lower bounds:
	\begin{equation*}
		\mathcal{G}^{(3)}(\delta) := \lim_{\tilde{\delta} \to \frac{\delta}{\mathbf{n}(1 - \delta)}} \left( \frac{2 B^{(3)}_{U,\tilde{\delta}}(\mathbf{n})}{\pi(1 - \mathbf{n}^2)} - \frac{2 P^{(3)}(\mathbf{n})}{\pi(1 - \mathbf{n}^2)} \right).
	\end{equation*}
	One can verify that as $\delta \to 0$, the gap  $\mathcal{G}^{(3)}(\delta)$ vanishes. This indicates that the convergence is pointwise in $\mathbf{n}$.
	To illustrate, setting $\delta = 0.1$ and $\tilde{\delta} = \frac{2\delta}{\mathbf{n}(1 - \delta)}$, the corresponding lower and upper bounds are illustrated in Figure~\ref{fig:lower and upper bound 3D}.
\end{enumerate}
\begin{figure}
	\centering
	\includegraphics[width=0.8\linewidth]{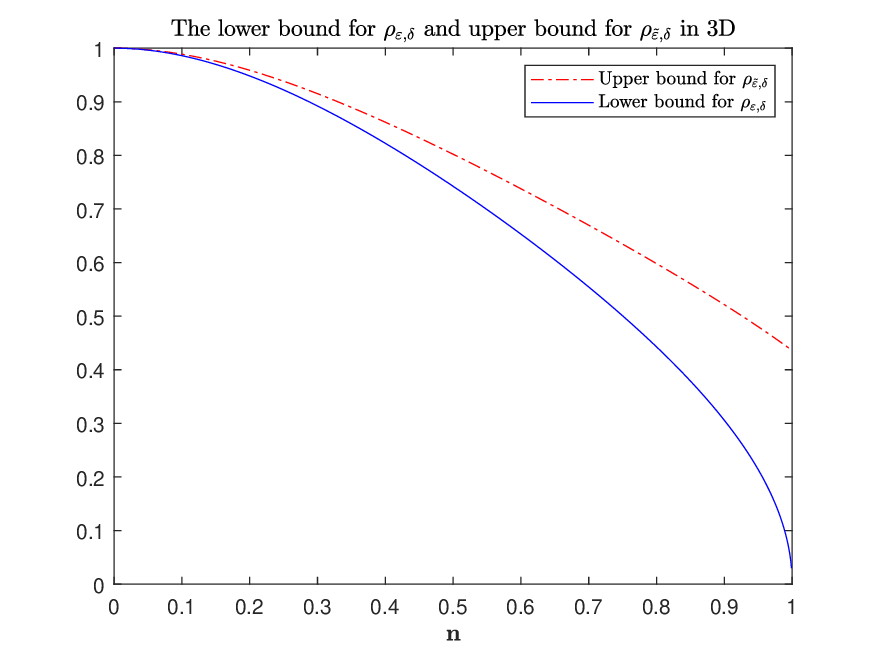}
	\caption{The graph of the lower bound for $\rho_{\varepsilon, \delta}$ and upper bound for $\rho_{\tilde{\varepsilon}, \delta}$ in 3D.}
	\label{fig:lower and upper bound 3D}
\end{figure}


\section*{Acknowledgments}
The work of Y. Jiang is supported in part by the National Key R\&D Program of China (2024YFA1012302). The work of H. Liu is supported by the Hong Kong RGC General Research Funds (Projects 11311122, 11300821, and 11304224), NSF/RGC Joint Research Fund (Project N\_CityU101/21) and the ANR/RGC Joint Research Fund (Project A\_CityU203/19). The work of K. Zhang is supported in part by China Natural National Science Foundation (Grant No. 12271207).


\end{document}